\documentclass{article}
\usepackage{amsfonts}
\usepackage{amsmath,amsthm}

\usepackage[all]{xy}
\usepackage{graphicx}

\usepackage{amssymb}
\usepackage{latexsym}

\DeclareMathAlphabet\EuFrak{U}{euf}{m}{n}	
\SetMathAlphabet\EuFrak{bold}{U}{euf}{b}{n}	

\newcommand{\lora}{\longrightarrow}

\newcommand{\hra}{\hookrightarrow}

\newcommand{\ovl}{\overline}
\newcommand{\unl}{\underline}
\newcommand{\wa}{\widehat}

\newcommand{\sC}{{\it C*}-}

\newcommand{\bC} {{\mathbb C}}
\newcommand{\bH} {{H}}

\newcommand{\bT} {{\mathbb T}}
\newcommand{\bU} {{\mathbb U}}
\newcommand{\bPU}{{\mathbb{PU}}}
\newcommand{\bZ} {{\mathbb Z}}
\newcommand{\bM} {{\mathbb M}}
\newcommand{\bN} {{\mathbb N}}
\newcommand{\bP} {{\mathbb P}}

\newcommand{\ud}{{{\mathbb U}(d)}}
\newcommand{\sud}{{{\mathbb {SU}}(d)}}

\newcommand{\rmd}{{\mathrm{d}}}

\newcommand{\eps}{\varepsilon}

\newcommand{\mA}{\mathcal A}
\newcommand{\mB}{\mathcal B}
\newcommand{\mC}{\mathcal C}

\newcommand{\mE}{\mathcal E}
\newcommand{\mF}{\mathcal F}
\newcommand{\mG}{\mathcal G}

\newcommand{\mK}{\mathcal K}
\newcommand{\mL}{\mathcal L}

\newcommand{\mO}{\mathcal O}

\newcommand{\mR}{\mathcal R}
\newcommand{\mS}{\mathcal S}
\newcommand{\mT}{\mathcal T}

\newcommand{\efb}{\EuFrak b}
\newcommand{\efg}{\EuFrak g}
\newcommand{\efn}{\EuFrak n}
\newcommand{\efq}{\EuFrak q}
\newcommand{\efu}{\EuFrak u}
\newcommand{\efv}{\EuFrak v}

\newcommand{\lto}{{{\mbox{\tiny $\stackrel{\to}{}$}}}}

\newcommand{\nN}{ G \lto N  }
\newcommand{\aG}{ G \lto {\bf aut}G }

\newcommand{\tend}{{\bf end}\mA}

\newcommand{\oro}{\mO_\rho}

\newcommand{\coe}{{\mO_\mE}}
\newcommand{\wE}{{S \mE}}

\newcommand{\ers}{{\mE^r , \mE^s}}
\newcommand{\rhors}{{\rho^r , \rho^s}}
\newcommand{\hrs}{{ \bH^r , \bH^s }}

\newcommand{\ii}{{\iota,\iota}}

\newcommand{\rr}{{\rho,\rho}}

\newcommand{\rs}{{\rho,\sigma}}
\newcommand{\sr}{{\sigma,\rho}}

\newcommand{\rsp}{{\rho',\sigma'}}

\newcommand{\sss}{{\sigma,\sigma}}

\newcommand{\obc}{{{\bf obj} \ \mC}}
\newcommand{\obt}{{{\bf obj} \ \mT}}

\newcommand{\sym} { {\bf{sym}} }

\newcommand{\mcSUE}{{ \mathcal{SUE} }}

\newcommand{\mcUE }{ \mathcal{UE} }

\newtheorem{thm}{Theorem}[section]
\newtheorem{cor}[thm]{Corollary}
\newtheorem{lem}[thm]{Lemma}
\newtheorem{prop}[thm]{Proposition}

\newtheorem{defn}[thm]{Definition}

\theoremstyle{definition}
\newtheorem{ex}{Example}[section]

\theoremstyle{remark}
\newtheorem{rem}{Remark}[section]

\numberwithin{equation}{section}

\begin{document}

\author{{\sf Ezio Vasselli}
                         \\{\it Dipartimento di Matematica}
                         \\{\it University of Rome "La Sapienza"}
			 \\{\it P.le Aldo Moro, 2 - 00185 Roma - Italy }
                         \\{\sf vasselli@mat.uniroma2.it}}

\title{Bundles of \sC categories, II: \\ \sC dynamical systems and Dixmier-Douady invariants}
\maketitle

\begin{abstract}
We introduce a cohomological invariant arising from a class in nonabelian cohomology. This invariant generalizes the Dixmier-Douady class and encodes the obstruction to a \sC algebra bundle being the fixed-point algebra of a gauge action. As an application, the duality breaking for group bundles vs. tensor \sC categories with non-simple unit is discussed in the setting of Nistor-Troitsky gauge-equivariant $K$-theory: there is a map assigning a nonabelian gerbe to a tensor category, and ``triviality'' of the gerbe is equivalent to the existence of a dual group bundle. At the \sC algebraic level, this corresponds to studying \sC algebra bundles with fibre a fixed-point algebra of the Cuntz algebra and in this case our invariant describes the obstruction to finding an embedding into the Cuntz-Pimsner algebra of a vector bundle.

\

{\em AMS Subj. Class.:} 18D10, 22D25, 14F05, 55N30.


\end{abstract}


\section{Introduction.}
\label{intro}

In a series of works in the last eighties, S. Doplicher and J.E. Roberts developed an abstract duality for compact groups, motivated by questions arised in the context of algebraic quantum field theory. In such a scenario, the dual object of a compact group is characterized as a tensor \sC category, namely a tensor category carrying an additional \sC algebraic structure (norm, conjugation). 

At the \sC algebraic level, one of the main discoveries in that setting  has been a machinery performing a duality theory for compact groups in the context of the Cuntz algebra (\cite{Cun77}). If $d \in \bN$ and $( \mO_d , \sigma_d )$ is the Cuntz \sC dynamical system (here $\sigma_d \in {\bf end} \mO_d$ denotes the canonical endomorphism, see \cite[\S 1]{DR87}), then every compact subgroup $G \subseteq \ud$ defines an automorphic action
\begin{equation}
\label{def_cag}
G \to {\bf aut} \mO_d
\ \ , \ \
G \ni g \mapsto \wa G \in {\bf aut} \mO_d
\ \ : \ \
\wa g (\psi_i) := \sum_{j=1}^d g_{ij} \psi_j
\ ,
\end{equation}
where $g_{ij} \in \bC$, $i,j = 1 , \ldots , d$, are the matrix elements of $g$ and $\{ \psi_i \}$ denotes the multiplet of mutually orthogonal partial isometries generating $\mO_d$. Let $\mO_G$ denote the fixed-point algebra of $\mO_d$ w.r.t. the action (\ref{def_cag}). Since $\sigma_d$ commutes with the $G$-action, the restriction $\sigma_G :=$ $\sigma_d |_{\mO_G} \in {\bf end} \mO_G$ is well-defined. The \sC dynamical system $(\mO_G,\sigma_G)$ allows one to reconstruct the following objects: (1) the group $G$, as the stabilizer of $\mO_G$ in ${\bf aut} \mO_d$; (2) The category $\wa G$ of tensor powers of the defining representation $G \hra \ud$, as the category $\wa \sigma_G$ with objects $\sigma_G^r$, $r \in \bN$, and arrows the {\em intertwiner spaces} of $\sigma_G$: 
\begin{equation}
\label{eq_sgrs} 
(\sigma_G^r,\sigma_G^s)
\ := \
\left\{
t \in \mO_G 
\ : \ 
\sigma^s(a) t = t\sigma_G^r(a)
,
a \in \mO_G
\right\}
\ \ , \ \
r,s \in \bN
\ .
\end{equation}
In this way, the map
\begin{equation}
\label{eq_galois}
G \mapsto ( \mO_G , \sigma_G )
\end{equation}
may be considered as a "Galois correspondence" for compact subgroups of $\ud$.

A more subtle question is when a \sC dynamical system $( \mA , \rho )$, $\rho \in {\bf end} \mA$, is isomorphic to $( \mO_G , \sigma_G )$ for some $G \subseteq \ud$. The solution to this problem (for $G$ contained in the special unitary group $\sud$) has been given in \cite[\S 4]{DR89A}: to get the above characterization, natural necessary conditions are the triviality of the centre of $\mA$ and the fact that $\mA$ is generated as a Banach space by the intertwiner spaces $(\rhors)$, $r,s \in \bN$; a more crucial condition is the existence of an intertwiner $\eps \in$ $( \rho^2 , \rho^2 )$, $\eps =$ $\eps^{-1}=$ $\eps^*$ (the {\em symmetry}), providing a representation $\bP_\infty \to \mA$ of the infinite permutation group and implementing suitable flips between elements of $( \rhors )$, $r,s \in \bN$. This structure is an abstract counterpart of the flip operator $\theta ( \psi \otimes \psi' ) :=$ $\psi' \otimes \psi$, $\psi , \psi' \in \bH$, where $\bH$ is the Hilbert space of dimension $d$.

In this way, a group $G \subseteq \sud$ is associated with $( \mA , \rho ,\eps )$ and the intertwiner spaces of $\rho$ are interpreted as $G$-invariant operators between tensor powers of $\bH$. In this sense $G$ is the {\em gauge group} associated with $( \mA , \rho , \eps )$, according to the motivation of Doplicher and Roberts (\cite{DR90}). The correspondence $( \mA , \rho , \eps ) \mapsto G$ is functorial: groups $G , G' \subseteq \sud$ are conjugates in $\ud$ if and only if there is an isomorphism $\alpha : ( \mA , \rho , \eps ) \to$ $( \mA' , \rho' , \eps' )$ of {\em pointed \sC dynamical systems}, in the sense that the conditions
$\alpha \circ \rho =$ $\rho' \circ \alpha$,
$\alpha (\eps) = \eps'$,
are fulfilled. As we shall see in the sequel, the previous conditions are equivalent to require an isomorphism of symmetric tensor \sC categories naturally associated with our \sC dynamical systems.

\

Our research program focused on the study of tensor \sC categories with non-simple unit. This means that the space of arrows of the identity object $\iota$ is isomorphic to an Abelian \sC algebra $C(X)$ for some compact Hausdorff space $X$. Thus the model category, rather than the one of Hilbert spaces, is the one of Hermitian vector bundles over $X$, that we denote by ${\bf vect}(X)$. In a previous work (\cite{Vas06}), we proved that every tensor \sC category with symmetry and conjugates can be regarded in terms of a bundle of \sC categories over $X$, with fibres duals of compact groups (see also \cite{Zit05}). By applying a standard technique, we associate pointed \sC dynamical systems of the type $( \mA , \rho , \eps )$ with objects of these categories; as a consequence of the above-mentioned results, each $( \mA , \rho , \eps )$ is a continuous bundle of \sC algebras with base $X$ and fibres pointed \sC dynamical systems $( \mO_{G_x} , \sigma_{G_x} , \theta_x )$, $x \in X$.

Starting from this result, it became natural to search for a classification of {\em locally trivial} pointed \sC dynamical systems $( \mA , \rho , \eps )$ with fibre $( \mO_G , \sigma_G , \theta)$, $G \subseteq \sud$. In the first paper of the present series, we gave such a classification in terms of the cohomology set $H^1(X,QG)$, $QG := NG \backslash G$, where $NG$ is the normalizer of $G$ in $\ud$ (\cite{Vas06}). In this way, $QG$-cocycles $\efq \in$ $H^1(X,QG)$ are put in one-to-one correspondence with pointed \sC dynamical systems $( \mO_\efq , \rho_\efq , \eps_\efq )$. From a different --but equivalent-- point of view, $H^1(X,QG)$ describes the set $\sym(X,\wa G)$ of isomorphism classes of "locally trivial" symmetric tensor \sC categories with fibre $\wa G$ and such that $(\ii) \simeq$ $C(X)$.

In the present paper we study the Galois correspondence (\ref{eq_galois}) and the associated abstract version in the case where $X$ is nontrivial. Instead of $\mO_d$, our reference algebra is the Cuntz-Pimsner algebra $\coe$ associated with the module of sections of a vector bundle $\mE \to X$, which yields a pointed \sC dynamical system $( \coe , \sigma_\mE , \theta_\mE )$. If $\mG \to X$ is a bundle of unitary automorphisms of $\mE$, then we can construct a pointed \sC dynamical system $( \mO_\mG , \sigma_\mG , \theta_\mE )$, $\mO_\mG \subseteq \coe$, from which it is possible to recover $\mG$ with the same method used for compact subgroups of $\ud$.
 
This leads to a duality for elements of $\sym(X,\wa G)$ vs. $G$-bundles acting on vector bundles in the sense of Nistor and Troitsky (\cite{NT04}). Anyway, what we get is not a generalization of the Doplicher-Roberts construction, as new phenomena arise. First, in general it is false that a category with fibre $\wa G$ is the dual of a $G$-bundle; the reason is a cohomological obstruction to the embedding into ${\bf vect}(X)$: in \sC algebraic terms, there are pointed \sC dynamical systems $( \mO_\efq , \rho_\efq , \eps_\efq )$ which do not admit an embedding into some $( \coe , \sigma_\mE , \theta_\mE )$. Secondly, an element of $\sym(X,\wa G)$ may be realized as the dual of non-isomorphic $G$-bundles: at the \sC algebraic level, we may get isomorphisms $( \mO_\efq , \rho_\efq , \eps_\efq ) \simeq$ $( \mO_\mG , \sigma_\mG , \theta_\mE )$, $( \mO_\efq , \rho_\efq , \eps_\efq ) \simeq$ $( \mO_{\mG'} , \sigma_{\mG'} , \theta_{\mE'} )$, with $\mG$ not isomorphic to $\mG'$ and $\mE$ not isomorphic to $\mE'$. In the present work we give a explanation of these facts in terms of properties of $H^1(X,QG)$, providing a complete geometrical characterization of $\sym(X,\wa G)$ for what concerns the duality theory.

\

The above-mentioned cohomological machinery has its roots in the general framework of principal bundles and can be applied to generic \sC algebra bundles. Let $G$ be a group of automorphisms of a \sC algebra $\mF_\bullet$ and $\mA_\bullet$ denote the fixed point algebra w.r.t. the $G$-action. It is natural to ask whether an $\mA_\bullet$-bundle $\mA$ admits an embedding into some $\mF_\bullet$-bundle. In general, the answer is negative and the obstruction is measured by a class
\begin{equation}
\label{def_dd}
\delta (\mA) \in H^2( X, {G'} ) 
\ ,
\end{equation}
where ${G'}$ is an Abelian quotient of $G$. When the above-mentioned embedding exists, $\mA$ is the fixed-point algebra w.r.t. a {\em gauge-action} of a group bundle $\mG \to X$ with fibre $G$ on an $\mF_\bullet$-bundle, in the sense of \cite{Vas07}. The above-mentioned obstruction for bundles with fibre $( \mO_G , \sigma_G , \theta )$ and the classical Dixmier-Douady invariant for bundles with fibre the compact operators (\cite[Ch.10]{Dix}), are particular cases of this construction.

\

The present work is organized as follows.

In \S\ref{sec_sp_cu} we recall some results relating pointed \sC dynamical systems with tensor \sC categories. Moreover, under the hypothesis that the inclusion $G \subseteq \ud$ is {\em covariant} (i.e., the embedding of $\wa G$ into the category of tensor powers of $\bH$ is unique up to unitary natural transformations), we give a geometrical characterization of the space of embeddings of $\mO_G$ into $\mO_d$ (Lemma \ref{cor_dual_od}) and a cohomological classification for $\sym(X,\wa G)$ (Thm.\ref{thm_amen2}). Note that every inclusion $G \subseteq \sud$ is covariant (in essence, this is proved in \cite[Lemma 6.7]{DR89}).

In \S\ref{sec_pb} we define some cohomological invariants for principal bundles. Given an exact sequence of topological groups $G \to NG \stackrel{p}{\to} QG$ and a space $X$, we consider the induced map of cohomology sets $p_* : H^1(X,NG) \to$ $H^1(X,QG)$ and construct a class $\delta (\efq) \in H^2(X,{G'})$ vanishing when $\efq$ is in the image of $p_*$. Moreover, a nonabelian $G$-gerbe $\breve \mG$ is associated with $\efq$, collapsing to a group bundle $\mG$ when $\efq$ is in the image of $p_*$.
Finally, for each $G \subseteq \sud$ we define a {\em Chern class} $c(\efq) \in H^2(X,\bZ)$ (Lemma \ref{lem_cherng}).

In \S\ref{sec_bu_co} we give some properties of gauge \sC dynamical systems and apply to them the construction of the previous section. In this way we construct the class (\ref{def_dd}), that we apply to pointed \sC dynamical systems (Lemma \ref{lem_qg}, Thm.\ref{thm_dd2}). The relation with the classical Dixmier-Douady invariant is discussed in Prop.\ref{ex_dd}.

In \S\ref{sec_gd} we prove a concrete duality for group bundles with fibre $G \subseteq \ud$. Let $\mE \to X$ be a rank $d$ vector bundle, $\wa \mE$ denote the category with objects the tensor powers $\mE^r$, $r \in \bN$, and arrows the spaces $(\ers)$ of bundle morphisms; then $\wa \mE$ is a symmetric tensor \sC category with $(\ii) \simeq C(X)$. We consider a group bundle $\mG \to X$ with a gauge action $\mG \times_X \mE \to \mE$ and define a symmetric tensor \sC subcategory $\wa \mG$ of $\wa \mE$, with arrows $\mG$-equivariant morphisms $(\ers)_\mG$, $r,s \in \bN$. We establish a one-to-one correspondence between tensor \sC subcategories of $\wa \mE$ and gauge actions (Prop.\ref{rem_dual}). Tensor \sC subcategories of $\wa \mE$ with fibre $\wa G$ are in one-to-one correspondence with reductions to $NG$ of the structure group of $\mE$ (Thm.\ref{thm_str_gr}): this yields a link between the categorical structure of $\wa \mE$ and the geometry of $\mE$.

In \S\ref{class} we discuss the breaking of abstract duality for categories $\mT$ with fibre $\wa G$. Isomorphism classes $[\mT] \in \sym(X,\wa G)$ such that there is an embedding $\eta : \mT \hra {\bf vect}(X)$ are in one-to-one correspondence with elements of the set 
$p_*(H^1(X,NG)) \subseteq H^1(X,QG)$ 
(Thm.\ref{thm_emb}). For each $\eta$ there is a vector bundle $\mE_\eta \to X$ and a $G$-bundle $\mG_\eta \to X$ acting on $\mE_\eta$ such that $\mT$ is isomorphic to $\wa \mG_\eta$. Applying the results of \S\ref{sec_pb}, we assign a class $\delta(\mT) \in H^2(X,{G'})$: if there is an embedding $\mT \to {\bf vect}(X)$ then $\delta (\mT)$ vanishes, and when such an embedding does not exist the role of the dual $G$-bundle is played by a $G$-gerbe (Thm.\ref{thm_ci}). Finally, we discuss the cases $G = \sud$ (Ex.\ref{ex_sud}, Ex.\ref{ex_su2}), $G = \bT$ (Ex.\ref{ex_T}) and $G = R_d$ (Ex.\ref{ex_rn}, $R_d$ denotes the group of roots of unity).

\section{Preliminaries.}

\subsection{Keywords and Notation.}
\label{sec_key}

Let $X$ be a locally compact Hausdorff space. If $\left\{ X_i \right\}$ is a cover of $X$, then we define $X_{ij} := X_i \cap X_j$, $X_{ijk} := X_i \cap X_j \cap X_k$. Moreover, we denote the \sC algebra of continuous functions on $X$ vanishing at infinity by $C_0(X)$; if $X$ is compact, then we denote the \sC algebra of continuous functions on $X$ by $C(X)$. If $U \subset X$ is open, then we denote the ideal in $C_0(X)$ (or $C(X)$) of functions vanishing in $X-U$ by $C_0(U)$. If $W \subset X$ is closed, then we define $C_W(X) :=$ $C_0(X-W)$; in particular, for every $x \in X$ we set $C_x(X) :=$ $C_0( X - \left\{ x \right\} )$. Since in the present paper we shall deal with \v Cech cohomology, we assume that every space has {\em good} covers (i.e. each $X_{ij}$, $X_{ijk}$, $\ldots$, is empty or contractible).

Let $\mA$ be a \sC algebra. We denote the set of automorphisms (resp. endomorphisms) of $\mA$, endowed with pointwise convergence topology, by ${\bf aut} \mA$ (resp. $\tend$). 
A pair $( \mA , \rho)$, with $\rho \in \tend$, is called \sC dynamical system. If $(\mA , \rho)$, $(\mA' , \rho')$ are \sC dynamical systems, then a \sC algebra morphism $\alpha : \mA \to \mA'$ such that $\alpha \circ \rho = \rho' \circ \alpha$ is 
denoted by $\alpha : ( \mA , \rho ) \to ( \mA' , \rho' )$. In particular, if $a \in \mA$, $a' \in \mA'$ and $\alpha (a) = a'$, then we write $\alpha : ( \mA , \rho , a ) \to $ $( \mA' , \rho' , a' )$ and refer to $\alpha$ as a morphism of {\em pointed \sC 
dynamical systems}. We denote the group of automorphisms of the pointed \sC dynamical system $(\mA , \rho , a)$ by ${\bf aut}(\mA , \rho , a)$.

Let $X$ be a locally compact Hausdorff space. A $C_0(X)${\em -algebra} is a \sC algebra $\mA$ endowed with a nondegenerate morphism from $C_0(X)$ into the centre of the multiplier algebra $M(\mA)$. It is customary to assume that such a morphism is injective, thus $C_0(X)$ will be regarded as a subalgebra of $M(\mA)$. For every $x \in X$, we define the {\em fibre epimorphism} as the quotient $\pi_x : \mA \to$ $\mA_x :=$ $\mA / (C_x(X)\mA)$ and call $\mA_x$ the fibre of $\mA$ over $x$. The group of $C_0(X)$-automorphisms of $\mA$ is denoted by ${\bf aut}_X \mA$. The {\em restriction} of $\mA$ on an open $U \subset X$ is given by the closed ideal obtained multiplying elements of $\mA$ by elements of $C_0(U)$, and is denoted by $\mA_U := C_0(U)\mA$.
We denote the (spatial) $C_0(X)$-tensor product by $\otimes_X$ (see \cite[\S 1.6]{Kas88}, where the notation ``$C(X)$'' is used to mean $C_0(X)$).
Examples of $C_0(X)$-algebras are continuous bundles of \sC algebras in the sense of \cite{KW95,Dix}; we refer to the last reference for the notion of {\em locally trivial} continuous bundle. Let $\mA_\bullet$ be a \sC algebra; to be concise, we will call {\em $\mA_\bullet$-bundle} a locally trivial continuous bundle of \sC algebras with fibre $\mA_\bullet$; to avoid confusion with bundles in the topological setting, we emphasize the fact that an $\mA_\bullet$-bundle is indeed a \sC algebra.

For standard notions about {\em vector bundles}, we refer to the classics \cite{Ati,Kar,Seg68}. In the present work, we will assume that every vector bundle is endowed with a Hermitian structure. We shall also consider {\em Banach bundles} (see \cite{Dup74},\cite[Ch.10]{Dix}).

For basic properties of fibre bundles and principal bundles, we refer to \cite[Ch.4,6]{Hus}, \cite[I.3]{Hir}.
If $p : Y \to X$ is a continuous map (i.e., a {\em bundle}), then we say that $p$ {\em has local sections} if for every $x \in X$ there is a neighbourhood $U \ni x$ and a continuous map $s : U \to Y$ such that $p \circ s = id_U$. If $p' : Y' \to X$ is a continuous map, then the {\em fibred product} is defined as the space
$Y \times_X Y' := \{ (y,y') \in Y \times Y' : p(y) = p'(y') \}$.
An expository introduction to nonabelian cohomology and gerbes is \cite{BS08}, where a good list of references is provided.

For basic properties of \sC categories and tensor \sC categories, we refer to \cite{DR89}. In particular, we make use of the terms {\em \sC functor, \sC epifunctor, \sC monofunctor, \sC isofunctor, \sC autofunctor} to denote functors preserving the \sC structure.

For every $r \in \bN$ we denote the permutation group of order $r$ by $\bP_r$ and the infinite permutation group by $\bP_\infty$, which is endowed with natural inclusions $\bP_s \subset \bP_\infty$, $s \in \bN$. For every $r,s \in \bN$, we denote the permutation exchanging the first $r$ objects with the remaining $s$ objects by $(r,s) \in \bP_{r+s}$.

\subsection{Bundles of \sC categories.}

A {\em \sC category} $\mC$ is a category having Banach spaces as sets of arrows and endowed with an involution $* : (\rs) \to (\sr)$, $\rs \in \obc$, such that the \sC identity $\left\| t^* \circ t \right\| =$ $\left\| t \right\|^2$, $t \in (\rs)$, is fulfilled. 
In this way, each $(\rr)$, $\rho \in \obc$, is a \sC algebra, whilst $(\rs)$ a Hilbert $(\sss)$-$(\rr)$-bimodule (see \cite{Vas06,KPW04}). 
In the present work we will consider \sC categories not necessarily endowed with {\em identity arrows} $1_\rho \in (\rr)$ (see \cite[\S 2.1]{Mit02}). In this setting, $(\rr)$ is not necessarily unital and we denote the multiplier algebra by $M(\rr)$.

Let $\mC$ be a \sC category and $X$ a locally compact Hausdorff space. $\mC$ is said to be a {\em $C_0(X)$-category} whenever there is a family $\left\{ i_\rho , \rho \in \obc\right\}$ of non-degenerate morphisms $i_\rho : C_0(X) \to M(\rr)$, called the $C_0(X)$-{\em structure}, such that 
\[
t \circ i_\rho(f) = i_\sigma (f) \circ t
\ \ , \ \
\rho , \sigma \in \obc 
\ , \ 
t \in (\rs)
\ , \ 
f \in C_0(X)
\ .
\]
The previous equality implies that each $(\rr)$, $\rho \in \obc$, is a $C_0(X)$-algebra. We assume that each $i_\rho$ is injective and write $ft := i_\sigma(f) \circ t$, $f \in C_0(X)$, $t \in (\rs)$. Functors preserving the $C_0(X)$-structure are called $C_0(X)$-{\em functors}.

If $U \subseteq X$ is open, then we define the {\em restriction on} $U$ as the \sC category $\mC_U$ having the same objects as $\mC$ and spaces of arrows $(\rs)_U :=$ $C_0(U) (\rs)$; note that $\mC_U$ may lacks identity arrows also when $\mC$ has identity arrows. If $W$ is closed, then we denote the \sC category having the same objects as $\mC$ and spaces of arrows the quotients $( \rs )_W :=$ $(\rs) \ / \left(  C_0(X-W) (\rs)  \right)$ by $\mC_W$; the corresponding \sC epifunctor $\pi_W : \mC \to \mC_W$ is called the {\em restriction functor}.
In particular, we define the {\em fibre} of $\mC$ over $x$ as $\mC_x :=$ $\mC_{\left\{ x \right\}}$ and call $\pi_x : \mC \to \mC_x$ the {\em fibre functor}. 
For every $\rs \in \obc$, $t \in (\rs)$, we define the {\em norm function} $n_t(x) :=$ $\left\| \pi_x (t) \right\|$, $x \in X$. It can be proved that $n_t$ is upper semicontinuous for each arrow $t$; when each $n_t$ is continuous, we say that $\mC$ is a {\em continuous bundle} over $X$. In this case, each $(\rs)$ is a continuous field of Banach spaces over $X$ and each $(\rr)$ is a continuous bundle of \sC algebras.

Let $\mC_\bullet$ be a \sC category. The {\em constant bundle} $X \mC_\bullet$ is the $C_0(X)$-category having the same objects as $\mC_\bullet$ and arrows the spaces $(\rs)^X$ of continuous maps vanishing at infinity from $X$ to $(\rs)$, $\rs \in \obc_\bullet$. 
A $C_0(X)$-category $\mC$ is said to be {\em locally trivial} whenever for each $x \in X$ there is an open neighbourhood $U \ni x$ with a $C_0(U)$-isofunctor 
$\alpha_U : \mC_U \to U \mC_\bullet$,
such that the induced map $\alpha_U : \obc \to {\bf obj} \ \mC_\bullet$ does not depend on the choice of $U$. The functors $\alpha_U$ are called {\em local charts}.

When $X$ is compact, the same constructions apply with the obvious modifications. 

%
%

\section{Tensor \sC categories and \sC dynamical systems.}
\label{sec_sp_cu}

The present section has two purposes. First, in order to make the present paper enough self-contained, we collect some results from \cite{DR87,Vas06} in a slightly different form and recall the notions of {\em special category} and {\em embedding functor}. Secondly, we describe the space of certain embedding functors in terms of a principal bundle (Lemma \ref{cor_dual_od}) and provide a classification result for bundles with fibre $\wa G$, $G \subset \ud$ (Thm.\ref{thm_amen2}); these results shall be applied in \S \ref{sec_gd}.

A {\em tensor} \sC category is a \sC category $\mT$ with identity arrows endowed with a \sC bifunctor $\otimes : \mT \times \mT \to$ $\mT$, called the {\em tensor product}. For brevity, we denote the tensor product of objects $\rs \in \obt$ by $\rho \sigma$, whilst the tensor product of arrows $t \in (\rs)$, $t' \in (\rsp)$, is denoted by $t \otimes t' \in$ $(\rho \rho' , \sigma \sigma')$.
We assume the existence of an {\em identity object} $\iota \in \obt$ such that $\iota \rho = \rho \iota = \rho$, $\rho \in \obt$: it can be easily verified that $(\ii)$ is an Abelian \sC algebra and every space of arrows $(\rs)$ is a Banach $(\ii)$-bimodule w.r.t. the operation of tensoring with arrows in $(\ii)$.

Let $X^\iota$ denote the spectrum of $\iota$; then $\mT$ is a $C(X^\iota)$-category in a natural way. In particular, it can be proved that $\mT$ is a continuous bundle if certain additional assumptions are satisfied (\cite{Zit05,Vas06}).

A tensor \sC category whose objects are $r$-fold tensor powers of an object $\rho$, $r \in \bN$, is denoted by $( \wa \rho , \otimes , \iota )$; for $r = 0$, we use the convention $\rho^0 := \iota$. In the sequel of the present work, we shall need to keep in evidence an arrow $a \in (\rhors)$ for some $r,s \in \bN$, so that we introduce the notation $( \wa \rho , \otimes , \iota , a )$. Moreover, we denote {\em tensor} \sC functors $\alpha : \wa \rho \to \wa \rho'$ (the term {\em tensor} means that $\otimes' \circ ( \alpha \times \alpha ) =$ $\alpha \circ \otimes$, $\alpha(\iota) =$ $\iota'$) such that $\alpha(a) = a'$ by
\[
\alpha : ( \wa \rho , \otimes , \iota , a ) \to ( \wa{\rho'} , \otimes' , \iota' , a' )
\ .
\]
If $( \mA , \rho , a)$ is a pointed \sC dynamical system, then the category $\wa \rho$ with objects the powers $\rho^r$, $r \in \bN$, and arrows the intertwiner spaces $(\rhors)$, $r,s \in \bN$, endowed with the tensor product
\[
\rho^r \rho^s := \rho^{r+s}
\ \ , \ \
t \otimes t' := t \rho^r(t')
\ \ , \ \
t \in (\rhors) \ , \ ( \rho^{r'} , \rho^{s'} )
\ ,
\]
is an example of such singly generated tensor \sC categories with a distinguished arrow. We denote the \sC algebra generated by the intertwiner spaces $(\rhors)$, $r,s \in \bN$, by $\oro$.

Actually, every tensor \sC category $( \wa \rho , \otimes , \iota )$ comes associated with a \sC dynamical system, in the following way (see \cite[\S 4]{DR89} for details). As a first step, we consider the maps $j_{r,s}(t) :=$ $t \otimes 1_\rho \in$ $( \rho^{r+1} , \rho^{s+1} )$, $t \in (\rhors)$ and define the Banach spaces 
$\oro^k :=$ $\lim_{\to r} ( (\rho^r,\rho^{r+k}) \ , \  j_{r,r+k} )$, $k \in \bZ$.
As a second step, we note that composition of arrows and involution induce a well-defined *-algebra structure on the direct sum $^0 \oro :=$ $\oplus_k \oro^k$. It can be proved that there is a unique \sC norm on $^0 \oro$ such that the {\em circle action} $\wa z (t) :=$ $z^k t$, $z \in \bT$, $t \in \oro^k$, extends to an automorphic action. In this way, the so-obtained \sC completion $\oro$ comes equipped with a continuous action $\bT \to {\bf aut} \oro$ with spectral subspaces $\oro^k$, $k \in \bN$, and also with a {\em canonical endomorphism} 
\[
\rho_* \in {\bf end} \oro 
\ \ , \ \
\rho_* (t) := 1_\rho \otimes t 
\ \ , \ \
t \in ( \rhors )
\ ,
\]
such that $\rho_* \circ \wa z = \wa z \circ \rho_*$, $z \in \bT$.
The pair $( \oro , \rho_* )$ is called the {\em DR-dynamical system associated with} $\rho$. {\em Since the maps $j_{r,s}$ are injective in all the cases of interest in the present work, in the sequel we will identify $t \in (\rhors)$ with the corresponding element of $\oro$}.

By construction we have $(\rhors) \subseteq$ $( \rho_*^r , \rho_*^s )$, $r,s \in \bN$. We say that $\rho$ is {\em amenable} if $(\rhors) =$ $( \rho_*^r , \rho_*^s )$, $r,s \in \bN$, and in that case $\wa \rho$ is said to be {\em amenably generated}. We summarize the above considerations in the following theorem, which also includes a reformulation of \cite[Prop.19]{Vas06}:
\begin{thm}
\label{thm_oro}
The map $( \wa \rho , \otimes , \iota , a ) \to$ $( \oro , \rho_* , a )$ defines a one-to-one correspondence between the class of amenably generated tensor \sC categories with a distinguished arrow and the class of pointed \sC dynamical systems $( \mA , \sigma , a )$ such that $\mA$ is generated by the intertwiner spaces of $\sigma$. Tensor \sC functors
$\alpha : 
( \wa \rho , \otimes , \iota , a ) \to 
( \wa{\rho'} , \otimes' , \iota' , a' )$ 
are in one-to-one correspondence with morphisms $\alpha : ( \oro , \rho_* , a ) \to$ $( \mO_{\rho'} , \rho'_* , a' )$ of pointed \sC dynamical systems. The category $\wa \rho$ is a continuous bundle over the spectrum $X^\iota$ of $( \ii )$ if and only if $\oro$ is a continuous bundle over $X^\iota$. If $\wa \rho$ is locally trivial as a bundle of \sC categories, then $\oro$ is locally trivial as a \sC algebra bundle.
\end{thm}

A tensor \sC category $( \mT , \otimes , \iota )$ is said to be {\em symmetric} if there is a family of unitary operators $\eps (\rs) \in$ $(\rho \sigma , \sigma \rho)$, $\rs \in \obt$, implementing the flips
\[
( t \otimes t' ) \circ \eps ( \rho' , \rho ) 
=
\eps ( \sigma' , \sigma ) \circ ( t' \otimes t )
\ .
\]
In particular, if $( \wa \rho , \otimes , \iota )$ is symmetric, then we define the {\em symmetry} operator
\[
\eps := \eps ( \rr ) \in ( \rho^2 , \rho^2 )
\ .
\]
It is well-known that $\eps$ induces a unitary representation of $\bP_\infty$, by considering products of the type
$\eps \circ ( 1_\rho \times \eps ) \circ ( 1_{\rho^r} \otimes \eps ) \circ \ldots$,
$r \in \bN$
(for example, see \cite[p.100]{DR87}). We denote the unitaries arising from such a representation by $\eps (p) \in (\rho^r, \rho^r)$, $r \in \bN$, $p \in$ $\bP_r \subseteq \bP_\infty$; in particular, we denote the unitary associated with $(r,s) \in \bP_{r+s}$ by
$\eps_\rho (r,s) \in (\rho^{r+s},\rho^{r+s})$.
If there is  
\[
\alpha : 
( \wa \rho    , \otimes  , \iota  , \eps )
\to 
( \wa{\rho'} , \otimes' , \iota' , \eps')
\ ,
\]
then the above considerations imply that
$\alpha (\eps_\rho(r,s)) = \eps'_{\rho'}(r,s)$,
$r,s \in \bN$.
We denote the pointed \sC dynamical system associated with $( \wa \rho    , \otimes  , \iota  , \eps )$ by $( \oro , \rho_* , \eps )$. According to the considerations of the previous section, we find that $\wa \rho$ is a $C(X^\iota)$-category. Now let $( \wa \rho_\bullet , \otimes_\bullet , \iota_\bullet , \eps_\bullet )$ be a symmetric tensor \sC category such that $( \iota_\bullet , \iota_\bullet ) \simeq \bC$; we denote the set of isomorphism classes of locally trivial symmetric tensor \sC categories $( \wa \rho , \otimes , \iota , \eps )$ having fibre $\wa \rho_\bullet$ and such that $(\ii) \simeq C(X)$ by
\[
\sym ( X , \wa \rho_\bullet )
\ .
\]
With the term {\em isomorphism}, here we mean a tensor \sC isofunctor of the type $\alpha : ( \wa \rho    , \otimes  , \iota  , \eps ) \to$ $( \wa{\rho'} , \otimes' , \iota' , \eps')$. 
A locally trivial symmetric tensor \sC category $( \wa \rho , \otimes , \iota , \eps )$ with fibre $\wa \rho_\bullet$ is called {\em $\wa \rho_\bullet$-bundle}; the class of $\wa \rho$ in $\sym ( X , \wa \rho_\bullet )$ is denoted by $[ \wa \rho , \otimes , \iota , \eps ]$ or, more concisely, by $[\wa \rho]$.

\begin{rem}
The condition $\alpha (\eps) = \eps'$ required in the previous notion of isomorphism comes from group duality. Let $G_1 , G_2$ be compact groups and $RG_1 , RG_2$ the associated symmetric tensor \sC categories of finite dimensional, continuous, unitary representations; if $\alpha : RG_1 \to  RG_2$ is an isomorphism of tensor categories, then a sufficient condition to get an isomorphism $\alpha^* : G_2 \to G_1$ is that $\alpha$ preserves the symmetry (see \cite{IK02}).
\end{rem}

The category $\bf hilb$ of Hilbert spaces, endowed with the usual tensor product, is clearly a symmetric tensor \sC category.
Of particular interest for the present work is the following class of subcategories of ${\bf hilb}$. Let $\bH$ be the standard Hilbert space of dimension $d \in \bN$; we denote the $r$-fold tensor power of $\bH$ by $\bH^r$ (for $r=0$, we define $\iota :=$ $\bH^0 :=$ $\bC$) and the space of linear operators from $\bH^r$ to $\bH^s$ by $(\hrs)$, $r,s \in \bN$; moreover, we consider the flip 
\[
\theta \in ( \bH^2 , \bH^2 )
\ .
\]
If $G \subseteq \ud$ is a compact group, then for every $g \in G$ we find that the $r$-fold tensor power $g_r$ is a unitary on $\bH^r$, so that we consider the spaces of $G$-invariant operators
\begin{equation}
\label{eq_ga_rs}
(\hrs)_G :=
\left\{ 
t \in ( \hrs )  :  t  =  \wa g (t)  :=  g_s \circ t \circ g_r^*  
\ , \ 
g \in G
\right\} \ .
\end{equation}
In particular, we have that $\theta \in (\bH^2,\bH^2)_G$. By defining the category $\wa G$ with objects $\bH^r$, $r \in \bN$, and arrows $(\hrs)_G$, we obtain a symmetric tensor \sC category $( \wa G , \otimes , \iota , \theta )$.  The pointed \sC dynamical system associated with $( \wa G , \otimes , \iota , \theta )$ in the sense of Thm.\ref{thm_oro} is $( \mO_G , \sigma_G , \theta )$, where $\mO_G$, $\sigma_G$ are defined in \S \ref{intro}. As mentioned in \S \ref{intro}, the category $\wa G$ is amenably generated, so that we have equalities
\begin{equation}
\label{eq_g_amen}
(\hrs)_G = ( \sigma_G^r , \sigma_G^s )
\ \ , \ \
r,s \in \bN
\ .
\end{equation}
If $G$ reduces to the trivial group, then we obtain the category $( \wa \bH , \otimes , \iota , \theta )$ of tensor powers of $\bH$ and Thm.\ref{thm_oro} yields the Cuntz \sC dynamical system $( \mO_d , \sigma_d , \theta )$. If $G = \ud$, then $(\hrs)_\ud$ is nontrivial only for $r=s$; in such a case, $(\bH^r,\bH^r)_\ud$ is generated as a vector space by the unitaries $\theta(p)$, $p \in \bP_r$.
Let $NG$ denote the normalizer of $G$ in $\ud$ and $QG := NG / G$ the quotient group; then, the map (\ref{def_cag}) induces an injective continuous action
\begin{equation}
\label{eq_qg}
QG \to {\bf aut}(\mO_G , \sigma_G , \theta ) 
\ \ , \ \
y \mapsto \wa y
\ .
\end{equation}
An other symmetric tensor \sC category that shall play an important role in the present paper is the category ${\bf vect} (X)$ with objects vector bundles over a compact Hausdorff space $X$ and arrows vector bundle morphisms.
\begin{defn}
\label{def_ef}
Let $( \mT , \otimes  , \iota , \eps )$ be a symmetric tensor \sC category. An {\bf embedding functor} is a \sC monofunctor $E : \mT \to {\bf vect} (X^\iota)$ preserving tensor product and symmetry.
\end{defn}

We now describe in geometrical terms a set of embedding functors of $\wa G$, $G \subseteq \ud$. To this end, let us denote the set of monomorphisms
\[
\eta : ( \mO_G , \sigma_G , \theta ) \to ( \mO_d , \sigma_d , \theta )
\]
by ${\bf emb}\mO_G$, and endow it with the pointwise norm topology; by Thm.\ref{thm_oro}, we can identify ${\bf emb}\mO_G$ with the set of embeddings 
$\beta : ( \wa G , \otimes , \iota , \theta ) \to ( \wa \bH , \otimes , \iota , \theta )$. 
In particular, we denote the group of autofunctors of the type
$\beta : ( \wa G , \otimes , \iota , \theta ) \to ( \wa G , \otimes , \iota , \theta )$
by ${\bf aut} \wa G$.

\begin{defn}
The faithful representation $G \subseteq \ud$ is said to be {\bf covariant} whenever for each $\eta \in {\bf emb}\mO_G$ there is $u \in \ud$ such that $\eta = \wa u |_{\mO_G}$.
\end{defn}

By Thm.\ref{thm_oro} the property of $G \subseteq \ud$ being covariant is equivalent to require that the inclusion functor $(\wa G,\otimes,\iota,\theta) \subseteq (\wa \bH,\otimes,\iota,\theta)$ is unique up to tensor unitary natural transformation. 
By \cite[Lemma 6.7,Thm.4.17]{DR89} (see also the following Thm.\ref{thm_amen}) every inclusion $G \subseteq \sud$ is covariant, thus we conclude that every compact Lie group has a faithful covariant representaton (in fact, it is well-known that every compact Lie group $G$ has a faithful representation $u : G \to \ud$, so it suffices to consider $u \oplus \ovl{\det u}$). Anyway there are interesting examples of covariant representations whose image is not contained in the special unitary group.

\begin{ex}
\label{ex_T0}
Let $G \subset \ud$ denote the image of $\bT$ under the action on $\bH \simeq \bC^d$ defined by scalar multiplication. Then $\wa G$ has spaces of arrows $(\hrs)_G = \delta_{rs} (\hrs)$, $r,s \in \bN$, where $\delta_{rs}$ denotes the Kronecker symbol. We have $NG = \ud$, $QG = \bPU(d)$. If $\eta \in {\bf emb}\mO_G$ then $\eta$ restricts to a \sC isomorphism $\eta : (\bH,\bH)_G = (\bH,\bH) \to (\bH,\bH)$, which is the inner automorphism induced by a unitary $u \in \ud$. Since $(\bH^r,\bH^r) \simeq \otimes^r(\bH,\bH)$ and $\eta(\otimes_i^r t_i) = \otimes_i^r \eta(t_i) = \otimes_i^r \wa u(t_i)$, $t_i \in (\bH,\bH)$, $i = 1 , \ldots , r$, we conclude that $\eta = \wa u|_{\mO_G}$ and $G \subseteq \ud$ is covariant.
\end{ex}

\begin{lem}
\label{cor_dual_od}
Let $G \subseteq \ud$ be covariant. Then ${\bf emb}\mO_G$ is homeomorphic to the coset space $\ud \backslash G$. For each locally compact Hausdorff space $Y$ and continuous map 
\begin{equation}
\label{eq_cbeta}
\beta : Y \to {\bf emb}\mO_G
\ \ , \ \
y \mapsto \beta_y
\ ,
\end{equation}
there is a finite open cover $\{ Y_l \}$ of $Y$ and continuous maps $u_l : Y_l \to \ud$ such that 
\begin{equation}
\label{eq_dual_od}
\wa u_{l,y} (t) = \beta_y (t)
\ \ , \ \
y \in Y_l
\ , \
t \in \mO_G
\ ,
\end{equation}
where $\wa u_{l,y} \in {\bf aut} \mO_d$ is defined by (\ref{def_cag}).
\end{lem}

\begin{proof}
We consider the fibration $q : \ud \to \ud \backslash G$ and define
\[
\chi : \ud \to {\bf emb}\mO_G
\ \ , \ \
u \mapsto \wa u |_{\mO_G}
\ .
\]
The map $\chi$ is clearly continuous and, since $G \subseteq \ud$ is covariant, it is also surjective. 
Now, (\ref{def_cag}) yields an isomorphism from $G$ to the stabilizer of $\mO_G$ in ${\bf aut}\mO_d$ (\cite[Cor.3.3]{DR87}), thus we find that $\chi (u_1) = \chi(u_2)$ if and only if $u_1^* u_2 \in G$, i.e. $q(u_1) = q(u_2)$. This proves that ${\bf emb}\mO_G$ is homeomorphic to $\ud \backslash G$.
Since $\ud$ is a compact Lie group, the map $q$ defines a principal $G$-bundle over $\ud \backslash G$, thus there is a finite open cover $\{ \Omega_l \}$ of $\ud \backslash G$ and local sections $s_l : \Omega_l \to \ud$, $q \circ s_l = id_{\Omega_l}$. Now, let us identify ${\bf emb}\mO_G$ with $\ud \backslash G$ and consider the map (\ref{eq_cbeta}); defining $Y_l := \beta^{-1}(\Omega_l)$ we obtain a finite open cover of $Y$ and set $u_{l,y} := s_l \circ \beta_y$, $y \in Y_l$. By definition of $\chi$, the equation (\ref{eq_dual_od}) is fulfilled and the theorem is proved.
\end{proof}

\noindent Let now $p : NG \to QG$ the natural projection. The following result is a version of \cite[Thm.36]{Vas06} for groups not necessarily contained in $\sud$:
\begin{thm}
\label{thm_amen2}
If $G \subseteq \ud$ is covariant then there is an isomorphism $QG \simeq {\bf aut} \wa G$, and for each compact Hausdorff space $X$ there is a bijective map $Q : \sym (X,\wa G) \to H^1(X,QG)$.
\end{thm}

\begin{proof}
Using Thm.\ref{thm_oro} we identify ${\bf aut} \wa G$ with ${\bf aut}(\mO_G,\sigma_G,\theta)$. By \cite[Lemma 32]{Vas06}, to prove the theorem it suffices to verify that (\ref{eq_qg}) is an isomorphism. 
Now, the same argument of the previous Lemma shows that if $\eta \in {\bf aut} \wa G$ then there is $u \in \ud$ such that $\wa u \in {\bf aut} \mO_d$ restricts to $\eta$ on $\mO_G$; since $\mO_G$ is $\wa u$-stable, for each $g \in G$ we find that $\wa u \circ \wa g \circ {\wa u}^{-1}$ is the identity on $\mO_G$, thus by \cite[Cor.3.3]{DR87} there is $g' \in G$ such that $g' = ugu^*$. We conclude that $u \in NG$ and since $\wa u \circ \wa g |_{\mO_G} = \wa u |_{\mO_G} = \eta$ for all $g \in G$ we find $\eta = \wa y$, where $y = p(u)$ and $\wa y$ is the image of $y \in QG$ under (\ref{eq_qg}).
\end{proof}

\begin{rem}
\label{rem_amen2}
Let ${\bf aut}(\mO_d;\mO_G)$ denote the group of automorphisms of 
$(\mO_d,\sigma,\theta)$
that restrict to elements of ${\bf aut}\mO_G$. The argument of the previous theorem shows that (\ref{def_cag}) induces the isomorphism $NG \to {\bf aut}(\mO_d;\mO_G)$, in such a way that for each $u \in NG$ we have
\[
\wa u |_{\mO_G} = \wa y 
\ \ , \ \ 
y := p(u) \in QG
\ .
\]
This yields a slight generalization of \cite[Thm.34]{Vas06}.
\end{rem}

Given a $\wa G$-bundle $( \wa \rho , \otimes , \iota , \eps )$, with $G \subseteq \ud$ covariant, we denote the associated class in $H^1(X,QG)$ by $Q[\wa \rho]$.

Now, let us consider a symmetric tensor \sC category $( \wa \rho , \otimes , \iota , \eps  )$. For every $n \in \bN$, we define the {\em antisymmetric projection}
\begin{equation}
\label{def_ap}
P_{\rho,\eps,n}
:= 
\frac{1}{n!} \sum_{p \in \bP(n)} {\mathrm{sign}}(p) \ \eps_\rho(p)
\ .
\end{equation}
The object $\rho$ is said to be {\em special} if there is $d \in \bN$ and a partial isometry $S \in (\iota,\rho^d)$ with support $P_{\rho,\eps,d}$, such that 
\begin{equation}
\label{eq_spce}
(S^* \otimes 1_\rho) \circ ( 1_\rho \otimes S ) = (-1)^{d-1} d^{-1} 1_\rho
\ \ \Leftrightarrow \ \ 
S^* \rho_*(S) = (-1)^{d-1} d^{-1} 1
\ .
\end{equation}
In such a case, $d$ is called the {\em dimension} of $\rho$. When $\rho$ is an endomorphism and is special in the above sense, we say that $\rho$ satisfies the {\em special conjugate property} (see \cite[\S 4]{DR89A}).

Special objects play a pivotal role in the Doplicher-Roberts theory. From the viewpoint of group duality they are an abstract characterization of the notion of representation with determinant $1$ (see \cite[\S 3]{DR89}). From the \sC algebraic point of view, they are an essential tool for the crossed product defined in \cite[\S 4]{DR89A}.

Let $G \subseteq \sud$. Then the object $\bH$ of $\wa G$ is special and has dimension $d$. In fact, we consider the isometry $S$ generating the totally antisymmetric tensor power $\wedge^d \bH$, and note that $u_d \circ S = \det u \cdot S = S$, $u \in \sud$, so that $S \in (\iota,\bH^d)_\sud \subseteq$ $(\iota,\bH^d)_G$ and (\ref{eq_spce}) follows from \cite[Lemma 2.2]{DR87}.
In particular when $G = \sud$ the spaces $(\hrs)_\sud$, $r,s \in \bN$, are generated by the operators $\theta(p)$, $p \in \bP_\infty$, and $S \in (\iota,\bH^d)_\sud$, by closing w.r.t. composition and tensor product.

\begin{defn}
\label{def_spec_cat}
A {\bf special category} is a locally trivial, symmetric tensor \sC category $( \wa \rho , \otimes , \iota , \eps )$ with fibre $( \wa \rho_\bullet , \otimes_\bullet , \iota_\bullet , \eps_\bullet )$, such that $\rho_\bullet$ is a special object.
\end{defn}
The dimension of the object $\rho$ generating the special category $\wa \rho$ is by definition the dimension of the special object $\rho_\bullet$ and is denoted by $d$. The main motivation of the present work is the search of embedding functors for special categories. The first step in this direction is given by the following classification result, proved in \cite[Thm.36]{Vas06}:
\begin{thm}
\label{thm_amen}
Let $( \wa \rho  , \otimes , \iota , \eps )$ be a special category with fibre $( \wa \rho_\bullet , \otimes_\bullet , \iota_\bullet , \eps_\bullet )$. Then: (1) $\rho$ is amenable; (2) Let $d \in \bN$ denote the dimension of $\rho$; then there is a compact Lie group $G \subseteq \sud$ such that $( \wa \rho_\bullet , \otimes_\bullet , \iota_\bullet , \eps_\bullet ) \simeq$ $( \wa G , \otimes , \iota , \theta )$; (3) There is a bijection $\sym ( X^\iota , \wa \rho_\bullet ) \simeq H^1(X^\iota , QG)$; (4) $\oro$ is an $\mO_G$-bundle.
\end{thm}

In general, the object generating a special category is not special. The obstruction to $\rho$ being special is encoded by the {\em Chern class} introduced in \cite[\S 3.0.3]{Vas06},
\begin{equation}
\label{def_cc_tso}
c (\rho) \in H^2 ( X^\iota , \bZ )
\ \ ,
\end{equation}
constructed by observing that the $(\ii)$-module $\mR_\rho :=$ $\{ \psi \in (\iota,\rho^d) : P_{\rho,\eps,d}\psi = \psi \}$ is the set of sections of a line bundle $\mL_\rho \to X^\iota$. The invariant $c(\rho)$ is defined as the first Chern class of $\mL_\rho$.

\section{Cohomology classes and principal bundles.}
\label{sec_pb}

In the present section we give an exact sequence and a cohomological invariant for a class of principal bundles. This elementary construction has important consequences in the setting of abstract duality for tensor \sC categories and can be regarded as a generalization of the Dixmier-Douady invariant.

Let $G$ a topological group with unit $1$ and $X$ a locally compact, paracompact Hausdorff space endowed with a (good) open cover $\{ X_i \}$. 
A $G${\em -cocycle} is given by a family $\efg := \{ g_{ij} \}$ of continuous maps $g_{ij} : X_{ij} \to G$ satisfying on $X_{ijk}$ the {\em cocycle relations} 
\[
g_{ij} g_{jk} = g_{ik}
\]
(which imply $g_{ij}g_{ji} = g_{ii} = 1$). In the sequel we will denote the evaluation of $g_{ij}$ on $x \in X_{ij}$ by $g_{ij,x}$. We say that $\efg$ is {\em cohomologous} to $\efg' := \{ g'_{ij} \}$ whenever there are maps $v_i : X_i \to G$ such that $g_{ij} v_j = v_i g_{ij}$ on $X_i$. This defines an equivalence relation over the set of $G$-cocycles, and passing to the inverse limit over open good covers provides the \v Cech {\em cohomology set} $H^1(X,G)$ (see \cite[I.3.5]{Kar}), which is a pointed set with distinguished element the class of the {\em trivial cocycle} ${\bf 1}$, $1_{ij,x} \equiv 1$. To be concise, sometimes in the sequel cocycles will be denoted simply by $\efg$ or $\{ g_{ij} \}$, and their classes in $H^1(X,G)$ by $[\efg]$ or $[ g_{ij} ]$. 
It is well-known that $H^1(X,G)$ classifies the principal $G$-bundles over $X$. When $G$ is Abelian, $H^1(X,G)$ coincides with the first cohomology {\em group} with coefficients in the sheaf $\mS_X (G)$ of germs of continuous maps from $X$ into $G$ (\cite[I.3.1]{Hir}).

We now pass to give a definition of {\em nonabelian \v Cech $2$-cohomology}. The basic object providing the coefficients of the theory is now given by a {\em crossed module} (also called {\em $2$--group}, see \cite[\S 3]{BS08}), which is defined by a morphism $i : G \to N$ of topological groups and an action $\alpha : N \to {\bf aut}G$, such that $i$ is equivariant for $\alpha$ and the adjoint actions $G \to {\bf aut}G$, $G \ni g \mapsto \wa g$, $N \to {\bf aut}N$, $N \ni u \mapsto \wa u$:
\[
\wa u \circ i = i \circ \alpha (u)
\ \ , \ \
\wa g = \alpha \circ i(g)
\ .
\]
The crossed module $(G,N,i,\alpha)$ is denoted for short by $G \to N$. To be concise we write $\unl g := i(g) \in N$, $g \in G$, and $\alpha (u) := \wa u$, $u \in N$. The equivariance relations ensure that no confusion will arise from this notation.

\begin{ex}
\label{ex_GNG}
Let $N$ be a topological group and $G$ a normal subgroup of $N$: then considering the inclusion $i : G \to N$ and the adjoint action $\alpha : N \to {\bf aut}G$, $u \mapsto \wa u$, yields a crossed module $G \to N$.
\end{ex}

A {\em cocycle pair} $\efb := (\efu,\efg)$ with coefficients in the crossed module $G \to N$ is given by families of maps 
\[
u_{ij} : X_{ij} \to N
\ \ , \ \
g_{ijk} : X_{ijk} \to G
\ ,
\]
satisfying the cocycle relations
\[
\left\{
\begin{array}{ll}
u_{ij} u_{jk} = {\unl g}_{ijk} u_{ik}
\\
g_{ijk} g_{ikl} = \wa u_{ij} (g_{jkl}) g_{ijl} \ ,
\end{array}
\right.
\]
where $\wa u_{ij} : X_{ij} \to {\bf aut}G$ is defined by means of $\alpha$. Cocycle pairs $\efb := (\efu,\efg)$, $\efb' := (\efu',\efg')$ are said to be {\em cohomologous} whenever there is a pair $(v,h)$ of families of maps
\[
v_i : X_i \to N
\ \ , \ \
h_{ij} : X_{ij} \to G
\ ,
\]
such that 
\[
\left\{
\begin{array}{ll}
v_i u'_{ij} = {\unl h}_{ij} u_{ij} v_j
\\
h_{ik} g_{ijk} = \wa v_i (g'_{ijk}) h_{ij} \wa u_{ij}(h_{jk})
\ .
\end{array}
\right.
\]
It can be proved that cohomology of cocycle pairs defines an equivalence relation (\cite[\S 4]{BS08}). The set of cohomology classes of cocycle pairs is by definition the {\em cohomology set} relative to the cover $\{ X_i \}$ {\em with coefficients in the crossed module} $G \to N$; passing to the limit w.r.t. covers yields the {\em \v Cech cohomology set} $\breve H^2 (X,\nN)$ with distinguished element the class of the trivial cocycle pair ${\bf 1} := (1,1)$, $1_{ij,x} := 1 \in N$, $1_{ijk,x} := 1 \in G$. The symbol $\breve H$ is used to emphasize that we deal with {\em nonabelian} cohomology sets. 
Note that our notation is not universally used in literature: sometimes the symbol $\breve H^1 (X,\nN)$ is used instead of $\breve H^2 (X,\nN)$ (see for example \cite{BS08}). The cohomology class of the cocycle pair $\efb = (\efu,\efg)$ is denoted by $[\efb] \equiv [\efu,\efg]$.

\begin{rem}
\label{rem_breve}
(1) Each $N$-cocycle $\efu := \{ u_{ij} \}$ defines the cocycle pair $\rmd \efu := ( \efu , {\bf 1} )$;
(2) If $G$ is Abelian and $\alpha$ is the trivial action, then each cocycle pair $( \efu , \efg )$ defines the cocycle $\efg = \{ g_{ijk} \}$ in the second (Abelian) cohomology of $G$.
\end{rem}

An important class of examples is the following: let $G$ be a topological group and $i : G \to {\bf aut}G$, $i(g) := \wa g$, denote the adjoint action; then taking $\alpha : {\bf aut}G \to {\bf aut}G$ as the identity map yields a crossed module $G \to {\bf aut}G$. Thus we can define the cohomology set $\breve H^2(X,\aG)$ with elements classes of cocycle pairs $(\lambda,\efg)$ of the type
\[
\lambda_{ij} : X_{ij} \to {\bf aut}G
\ \ , \ \
g_{ijk} : X_{ijk} \to G
\ \ : \ \
\left\{
\begin{array}{ll}
\lambda_{ij} \lambda_{jk} = \wa g_{ijk} \lambda_{ik}
\\
g_{ijk} g_{ikl} = \lambda_{ij} (g_{jkl}) g_{ijl} \ ,
\end{array}
\right.
\]
where each $\wa g_{ijk} : X_{ijk} \to {\bf aut}G$ is defined by adjoint action.

\begin{rem}
\label{rem_def_gerbe}
According to the considerations in \cite[\S 2]{BS08}, $\breve H^2 (X,\aG)$ classifies the $G$-gerbes on $X$ up to isomorphism. In the present paper we use the term {\bf $G$-gerbe} to mean a principal $2$--bundle over $X$ with fibre the crossed module $G \to {\bf aut}G$. In this way, cocycle pairs with coefficients in $G \to {\bf aut}G$ are interpreted as transition maps for $G$-gerbes, and $G$-bundles define $G$-gerbes such that the associated cocycle pairs are of the type $\rmd \lambda = ( \lambda,{\bf 1} )$, $\lambda \in H^1(X,{\bf aut}G)$ (see Rem.\ref{rem_breve}).
\end{rem}

We define the maps
\begin{equation}
\label{eq_gamma}
\gamma_* : H^1(X,N) \to H^1(X,{\bf aut}G)
\ \ , \ \
[ u_{ij} ] \mapsto [ \wa u_{ij} ]
\ ,
\end{equation}
\begin{equation}
\label{eq_bgamma}
\breve \gamma_* : \breve H^2(X,G \lto N) \to \breve H^2(X,\aG)
\ \ , \ \
[ \{ u_{ij} \} , \efg ] \mapsto [ \{ \wa u_{ij} \} , \efg ]
\ .
\end{equation}

Let now $N$ denote a topological group and $G \subseteq N$ a normal subgroup. Defining $QG := N \backslash G$ yields the exact sequence
\begin{equation}
\label{def_p}
1 \to G \stackrel{i}{\hra}  N \stackrel{p}{\to}  QG \to 1 \ .
\end{equation}
Let $N_G$ denote the smaller normal subgroup of $N$ containing the set
\[
[G,N] \ := \  \{ gug^{-1}u^{-1} \ , \ g \in G , u \in N \} \ .
\]
We consider the quotient map $\pi_N : N \to {N'} := N \backslash N_G$ and define ${G'} := \pi_N(G)$. Note that by construction ${G'}$ is Abelian; when $G$  is contained in the centre of $N$ we have that $[G,N]$ is trivial and $N_G = \{ 1 \}$, $N=N'$, $G=G'$.

\begin{lem}
\label{lem_Gab}
Let $G$ be a normal subgroup of the topological group $N$ and suppose that the fibration $p : N \to QG := N \backslash G$ has local sections. Then for every locally compact, paracompact Hausdorff space $X$ there is an isomorphism of pointed sets
\begin{equation}
\label{eq_h1_bh1}
\nu : H^1(X,QG) \to \breve H^2(X,G \lto N)
\ .
\end{equation}
Moreover, there is a commutative diagram
\begin{equation}
\label{eq_cd_ab}
\xymatrix{
   G
   \ar[r]^-{ i }
   \ar[d]_-{ \pi_G }
 & N
   \ar[d]^-{ \pi_N }
\\ {G'}
   \ar[r]^-{ i' }
 & {N'}
}
\end{equation}
which yields the map
\begin{equation}
\label{eq_pn}
\pi_{N,*} : \breve H^2(X, G \lto N ) \to H^2(X,{G'}) \ .
\end{equation}
\end{lem}

\begin{proof}
The fact that there is an isomorphism as in (\ref{eq_h1_bh1}) is proved in \cite[Lemma 2]{BS08}, anyway for the reader's convenience we give a sketch of the proof.
Let $\efq := \{ y_{ij} \}$ be a $QG$-cocycle; since $p$ has local sections, up to performing a refinement of $\{ X_i \}$ there are maps $u_{ij} : X_{ij} \to N$ such that $y_{ij} = p \circ u_{ij}$ (it suffice to define $u_{ij} := s \circ y_{ij}$, where $s : U \to N$, $U \subseteq QG$, $y_{ij}(X_{ij}) \subseteq U$, is a local section). Since $p \circ ( u_{ij} u_{jk} u_{ik}^{-1} ) = y_{ij} y_{jk} y_{ki} = 1$, we conclude that there is $g_{ijk} : X_{ijk} \to G$ such that 
$ u_{ij} u_{jk} = g_{ijk} u_{ik}$.
It is trivial to check that $( \{ u_{ij} \} , \{ g_{ijk} \} )$ is a cocycle pair, and we define
\begin{equation}
\label{def_gjik}
\nu [\efq] := [ \{ u_{ij} \} , \{ g_{ijk} \} ] 
\ \ , \ \
[\efq] :=  [ y_{ij} ] \in H^1(X,QG)
\ .
\end{equation}
On the other side, if $\efb := ( \{ u_{ij} \} , \{ g_{ijk} \} )$ is a cocycle pair then defining $p_*[\efb] := [ p \circ u_{ij} ]$ yields an inverse of $\nu$.
We now prove (\ref{eq_pn}). Defining $\pi_G := \pi_N |_G$ yields the commutative diagram (\ref{eq_cd_ab}); if $\efb := ( \efu , \efg )$, $\efg := \{ g_{ijk} \}$, is a cocycle pair, then we define 
$\pi_{N,*}[\efb] := [ \pi_N \circ g_{ijk} ]$ 
and this yields the desired map (note in fact that $\pi_N \circ \wa u_{ij}(g_{jkl}) = \pi_N \circ g_{jkl}$, so that $\{ \pi_N \circ g_{ijk} \}$ is a 2--${G'}$-cocycle).
\end{proof}

Now, by functoriality of $H^1(X, \ \cdot \ )$ there is a sequence of maps of pointed sets
\begin{equation}
\label{eq_pH}
H^1(X,G) \stackrel{i_*}{\lora}  H^1(X,N) \stackrel{p_*}{\lora} H^1(X,QG) \ .
\end{equation}
In fact, $p_* \circ i_*[\efg] = [{\bf 1}]$ for each $\efg \in H^1(X,G)$. In the following result we give an obstruction to $p_*$ being surjective.

\begin{lem}
\label{lem_ddc}
Let $G$ be a normal subgroup of the topological group $N$ such that the fibration $p : N \to QG := N \backslash G$ has local sections. Then we have the following sequence of maps of pointed sets:
\begin{equation}
\label{eq_main}
\xymatrix{
   H^1(X,G)
   \ar[r]^-{ i_* }
 & H^1(X,N)
   \ar[r]^-{ p_* }
   \ar[d]_-{ \gamma_* }
 & H^1(X,QG) 
   \ar[r]^-{ \delta }
   \ar[d]^-{ \breve \gamma_* \circ \nu }
 & H^2(X,{G'})
\\ {} 
 & H^1(X,{\bf aut}G )
   \ar[r]^-{ \rmd_* }
 & \breve H^2 ( X , \aG )
 & {}
}
\end{equation}
Here $\rmd_*$ is induced by the map $\rmd \lambda := (\lambda,{\bf 1})$, and the square is commutative. When $G$ is contained in the centre of $N$, the upper horizontal row is exact and $G'=G$.
\end{lem}

\begin{proof}
The proof of the Lemma is based on the maps introduced in Lemma \ref{lem_Gab}.
Define $\delta := \pi_{N,*} \circ \nu$. If $\efu := \{ u_{ij} \}$ is an $N$-cocycle and $\efq := \{ y_{ij} := p \circ u_{ij} \}$ then by definition of $\nu$ we find $\nu[\efq] = [ \efu , {\bf 1} ] = \rmd_*[\efu]$ (see Rem.\ref{rem_breve}); moreover, $\delta[\efq] = \pi_{N,*} [ \efu , {\bf 1}] = [{\bf 1}]$ and this proves that $p_*(H^1(X,N)) \subseteq \ker \delta$. 
We now prove that the square is commutative. To this end, note that for each $N$-cocycle $\efu := \{ u_{ij} \}$ we find 
$\rmd_* \circ \gamma_* [\efu] = 
 [ \rmd \{ \wa u_{ij} \} ] = 
 [ \{ \wa u_{ij} \} , {\bf 1} ]$; 
on the other side, if $\efq := \{ y_{ij} \} := \{ p \circ u_{ij} \}$ then
$\breve \gamma_* \circ \nu \circ p_* [\efu] = 
 \breve \gamma_* \circ \nu [\efq] = 
 \breve \gamma_* [ \efu , {\bf 1} ] = 
 [ \{ \wa u_{ij} \} , {\bf 1} ]$, 
and we conclude that the square is commutative.
Finally, we prove that the upper horizontal row is exact when $G$ is contained in the centre of $N$; to this end, it suffices to verify that $\ker \delta \subseteq p_*(H^1(X,N))$. 
Now, we have $G=G'$ and the map $\pi_{N,*}$ takes the form $\pi_{N,*} [\efu,\efg] := [\efg]$. Since $\nu$ is bijective we have that $\delta[\efq] = [{\bf 1}]$ if and only if $\pi_{N,*}[\efu,\efg] = [{\bf 1}]$, where $[\efu,\efg] = \nu[\efq]$. This means that $\efg = \{ g_{ijk} \}$ is a trivial $G$-$2$-cocycle, so that there are maps $h_{ij} : X_{ij} \to G$ such that $h_{ij} h_{jk} = g_{ijk} h_{ik}$; the pair $( 1 , \{ h_{ij} \} )$ defines a $2$-cocycle equivalence between $(\efu,\efg)$ and $(\efu',1)$, where $\efu' := \{ u_{ij} h_{ji} \}$ is, by construction, an $N$-cocycle. By definition of $\nu$ we have $p_*[\efu'] = [\efq]$, and this proves
$p_*(H^1(X,N)) = \ker \delta$.
Thus the upper horizontal row is exact as desired.
\end{proof}

Note that by classical results when $N$ is a compact Lie group and $G$ is closed, $G$, $QG$, ${G'}$ are compact Lie groups and the fibration $N \to QG$ has local sections.

An interesting class of examples is the following. Let $\bU$ be the unitary group of an infinite dimensional Hilbert space; then, the centre of $\bU$ is the torus $\bT$ and $\bPU := \bU / \bT$ is the projective unitary group. In this case, $\delta$ takes the form
\begin{equation}
\label{eq_dd_class}
\left\{
\begin{array}{ll}
\delta : H^1(X,\bPU) \to H^2(X,\bT) \simeq H^3(X,\bZ)
\\
\delta [\efq] := [\efg] 
\ , \
\efq := \{ y_{ij} \} 
\ , \
\efg := \{ g_{ijk} \}
\end{array}
\right.
\end{equation}
(where $\{ g_{ijk} \}$ is defined by (\ref{def_gjik})) and it is well-known that it is an isomorphism (see \cite[\S 10.7.12]{Dix} and following sections).

In the following Lemma, we define a Chern class for a $QG$-cocycle when $G \subseteq \sud$ and $NG$ is the normalizer of $G$ in $\ud$.
\begin{lem}
\label{lem_cherng}
Let $G \subseteq \sud$. Then there is a map $c : H^1(X,QG) \to H^2 (X,\bZ)$; if $\efq$ is a trivial $QG$-cocycle then $c[\efq] = 0$.
\end{lem}

\begin{proof}
It suffice to note that the determinant defines a group morphism $\det : NG \to \bT$. Since $G \subseteq \sud$, we find that $\det$ factorizes through a morphism $\det_Q : QG \to \bT$. The functoriality of $H^1 ( X , \ \cdot \ )$, and the well-known isomorphism $H^1(X,\bT) \simeq H^2(X,\bZ)$, complete the proof.
\end{proof}

\section{Bundles of \sC algebras and cohomology classes.}
\label{sec_bu_co}

In the present section we give an application of the cohomology class $\delta$ defined in Lemma \ref{lem_ddc} to bundles of \sC algebras. To this end, in the following lines we present some constructions involving principal bundles and \sC dynamical systems.

Let $\mF_\bullet$ be a \sC algebra and $X$ a locally compact, paracompact Hausdorff space. Then the cohomology set $H^1(X,{\bf aut}\mF_\bullet)$ can be interpreted as the set of isomorphism classes of $\mF_\bullet$-bundles, in the following way: for each ${\bf aut}\mF_\bullet$-cocycle $\efu :=$ $\{ u_{ij} \}$, denote the fibre bundle with fibre $\mF_\bullet$ and transition maps $\{ u_{ij} \}$ by 
\[
\pi : \wa \mF \to X
\]
(see \cite[5.3.2]{Hus}); by construction, $\wa \mF$ is endowed with local charts
\[
\pi_i : \wa \mF |_{X_i} := \pi^{-1}(X_i) \to X_i \times \mF_\bullet \ ,  
\]
where $\{ X_i \}_i$ is an open cover of $X$, in such a way that
\begin{equation}
\label{eq_lc_c}
\pi_i \circ \pi_j^{-1} (x,v_\bullet) \ = \ ( x , u_{ij,x} (v_\bullet) )
\ \ , \ \
x \in X_{ij} \ , \ v_\bullet \in \mF_\bullet
\ .
\end{equation}
The set of sections 
$t : X \to \wa \mF$, $p \circ t = id_X$,
such that the norm function $\{ X \ni x \mapsto \left\| t(x) \right\| \}$ vanishes at infinity has a natural structure of $\mF_\bullet$-bundle, that we denote by $\mF_\efu$. 
%
%
%
On the converse, given an $\mF_\bullet$-bundle $\mF$, using the method exposed in \cite[\S 3.1]{Vas07} (see also the related references), we can construct a fibre bundle $\pi : \wa \mF \to X$ with fibre $\mF_\bullet$, in such a way that $\mF$ is isomorphic to the $C_0(X)$-algebra of sections of $\wa \mF$. The correspondence $\mF \mapsto \wa \mF$ is functorial: $C_0(X)$-morphisms $\tau : \mF_1 \to \mF_2$ correspond to bundle morphisms $\wa \tau : \wa \mF_1 \to \wa \mF_2$ such that $\tau (t) = \wa \tau \circ t$, $t \in \mF_1$.

Let $\mF_1$, $\mF_2$ be $\mF_\bullet$-bundles and $K$ a subgroup of ${\bf aut}\mF_\bullet$; a $C_0(X)$-isomorphism $\beta : \mF_1 \to \mF_2$ is said to be {\em $K$-equivariant} if there is an open cover $\{ X_i \}_{i \in I}$ trivializing $\wa \mF_1$, $\wa \mF_2$ by means of local charts $\pi_{i,k} : \wa \mF_k |_{X_i} \to X_i \times \mF_\bullet$, $k = 1,2$, $i \in I$, with automorphisms $\beta_{i,x} \in K$, $i \in I$, $x \in X_i$, satisfying
\[
\wa \beta \circ \pi_{i,1}^{-1} ( x,v_\bullet )
\ = \
\pi_{i,2}^{-1}  ( x , \beta_{i,x} (v_\bullet) )
\ \ , \ \
v_\bullet \in \mF_\bullet
\]
(roughly speaking, at the local level $\beta$ is described by automorphisms in $K$). In such a case, we say that $\mF_1$ is $K$-$C_0(X)$-{\em isomorphic} to $\mF_2$. Moreover, we say that an $\mF_\bullet$-bundle $\mF$ has {\em structure group} $K$ if $\mF = \mF_\efu$ for some $K$-cocycle $\efu$. It is easy to verify that $K$-cocycles $\efu$, $\efv$ are equivalent in $H^1(X,K)$ if and only if the associated $\mF_\bullet$-bundles are $K$-$C_0(X)$-isomorphic.

\begin{rem}
\label{rem_lt_pds}
Let $( \mA_\bullet , \rho_\bullet , a_\bullet )$ be a pointed \sC dynamical system and $K :=$ ${\bf aut} ( \mA_\bullet , \rho_\bullet , a_\bullet ) \subseteq$ ${\bf aut} \mA_\bullet$. An $\mA_\bullet$-bundle $\mA$ has structure group $K$ if and only if there is $\rho \in {\bf end}_X \mA$ and $a \in \mA$ with local charts $\pi_i : \wa \mA |_{X_i} \to X_i \times \mA_\bullet$, such that 
\[
\pi_i \circ \wa \rho (v)  =   \rho_\bullet (v_\bullet)
\ \ , \ \
\pi_i (a) =  ( x , a_\bullet )
\ \ , \ \
v \in \wa \mA \ , \ ( x,v_\bullet ) := \pi_i (v)
\ .
\]
In this case, we say that $( \mA , \rho , a )$ is a {\bf locally trivial pointed \sC dynamical system}. Now, $\beta : \mA \to \mA'$ is a $K$-$C_0(X)$-isomorphism if and only if $\beta$ is an isomorphism of pointed \sC dynamical systems. So that, $H^1(X,K)$ describes the set of isomorphism classes of locally trivial pointed \sC dynamical systems $( \mA , \rho , a )$ with fibre $( \mA_\bullet , \rho_\bullet , a_\bullet )$. 
In the sequel, we shall make use of the following fact: if $N$ is a subgroup of $K$ and $\efn$ is an $N$-cocycle, then we may regard $\efn$ as a $K$-cocycle; thus, if $\mA$ is an $\mA_\bullet$-bundle with structure group $N$, then $\mA$ defines a locally trivial pointed \sC dynamical system $( \mA , \rho , a )$.
\end{rem}

The next lemma is an application of the previous ideas.
\begin{lem}
\label{lem_qg}
Let $d \in \bN$, $G \subseteq \ud$ be covariant and $QG := NG \backslash G$. Then for each compact Hausdorff space $X$ there are one-to-one correspondences between: 
(1) $QG$-cocycles; 
(2) locally trivial pointed \sC dynamical systems with fibre $( \mO_G , \sigma_G , \theta )$; 
(3) $\wa G$-bundles.
\end{lem}

\begin{proof}
Consider the pointed \sC dynamical system $( \mO_G , \sigma_G , \theta )$ with the action (\ref{eq_qg}), then apply Rem.\ref{rem_lt_pds}, Thm.\ref{thm_amen2} and Thm.\ref{thm_oro}.
\end{proof}

The following construction may be regarded as an analogue of the notion of group action in the setting of \sC bundles and appeared in \cite[\S 3.2]{Vas07}. Let $G$ be a subgroup of ${\bf aut}\mF_\bullet$. A {\em gauge \sC dynamical system with fibre $( \mF_\bullet , G )$} is given by a triple $( \mF , \mG , \alpha )$, where $\mF$ is an $\mF_\bullet$-bundle, $\eta : \mG \to X$ is a bundle with fibre $G$ and 
\[
\alpha : \mG \times_X \wa \mF \to \wa \mF
\]
is a continuous map such that for each $x \in X$ there is a neighbourood $U$ of $x$ with local charts 
\begin{equation}
\label{eq_ltF}
\eta_U : \mG|_U \to U \times G
\ \ , \ \
\pi_U : \wa \mF|_U \to U \times \mF_\bullet
\ ,
\end{equation}
satisfying
\begin{equation}
\label{eq_gact}
\pi_U \circ \alpha ( y , v ) = ( x , y_\bullet (v_\bullet) )
\ ,
\end{equation}
where 
\[
x := \pi(v) = \eta(y)
\ \ , \ \
(x,y_\bullet) := \eta_U (y)
\ \ , \ \
(x,v_\bullet) := \pi_U  (v)
\]
(so that $y_\bullet \in G \subseteq {\bf aut}\mF_\bullet$ and $v_\bullet \in \mF_\bullet$). We say that $( \mF , \mG , \alpha )$ has structure group $K$ if $\mF$ has structure group $K$.

Usual continuous actions are related with gauge \sC dynamical systems in the following way: if $S$ is a set of sections of $\mG$ which is also a group w.r.t. the operations defined pointwise, then there is an action $S \to {\bf aut}_X \mF$; in particular, every continuous action $G \to {\bf aut}_X \mF$ can be regarded as a gauge action on $\mF$ of the bundle $\mG := X \times G$ (see \cite[\S 3.2]{Vas07} for details).

The {\em fixed-point algebra} of $( \mF , \mG , \alpha )$ is given by the $C_0(X)$-algebra
\[
\mF^\alpha
:=
\{
t \in \mF : \alpha ( y , t(x) ) = t(x) \ \ , \ \ x \in X , y \in \eta^{-1} (x)
\}
\ .
\]
Let $\mA_\bullet \subseteq \mF_\bullet$ denote the fixed-point algebra w.r.t. the $G$-action. Then (\ref{eq_gact}) implies that $\mF^\alpha$ is an $\mA_\bullet$-bundle.

We now expose the main construction of the present section.
Again, we consider a \sC algebra $\mF_\bullet$ and a subgroup $K$ of ${\bf aut}\mF_\bullet$; moreover, we pick a subgroup $G$ of $K$ and denote the fixed-point algebra w.r.t. the $G$-action by $\mA_\bullet$. We consider the normalizer of $G$ in $K$ and the associated quotient group, as follows:
\begin{equation}
\label{def_gqg}
\left\{
\begin{array}{ll}
NG  := \{ u \in K : u \circ g \circ u^{-1} \in G  \}
\\
p : NG \to QG := NG \backslash G
\ .
\end{array}
\right.
\end{equation}
By construction, for every $u \in NG$, $g \in G$, $a \in \mA_\bullet$ there is $g' \in G$ such that
$g \circ u (a) =$ $u \circ g' (a) = u (a)$.
The above equalities imply that the $NG$-action on $\mF_\bullet$ factorizes through a $QG$-action
\begin{equation}
\label{eq_QG_act}
QG \to {\bf aut} \mA_\bullet 
\ \ , \ \
p(u) \mapsto u |_{\mA_\bullet}
\ , \
u \in NG
\ ;
\end{equation}
thus, applying the above procedure, for every $QG$-cocycle $\efq$ we can construct an $\mA_\bullet$-bundle $\mA_\efq$.
\begin{lem}
\label{lem_dd}
Let $\efq :=$ $( \{X_i \} , \{ y_{ij} \} ) \in$ $H^1(X,QG)$ and $\mA_\efq$ denote the associated $\mA_\bullet$-bundle. Then the following are equivalent:
\begin{enumerate}
\item There is a gauge \sC dynamical system $(\mF,\mG,\alpha)$ with fibre $(\mF_\bullet,G)$ and structure group $NG$, such that $\mA_\efq$ is $QG$-$C_0(X)$-isomorphic to $\mF^\alpha$;
\item there is an $NG$-cocycle $\efn$ such that $[\efq] = p_*[\efn]$, where $p_* : H^1(X,NG) \to H^1(X,QG)$ is the map induced by (\ref{def_gqg}(2)).
\end{enumerate}
\end{lem}

\begin{proof}
(1) $\Rightarrow$ (2): Let us denote the $NG$-cocycle associated with $\mF$ by $\efn :=$ $( \{ X_i \} , \{ u_{ij} \} )$. We assume that $\{ X_i \}$ trivializes $\mG$ and $\wa \mF$ (otherwise, we perform a refinement of $\{ X_i \}$), so that we have local charts $\eta_i : \mG |_{X_i} \to X_i \times G$, $\pi_i : \wa \mF |_{X_i} \to X_i \times \mF_\bullet$ fulfilling (\ref{eq_gact}), with $\{ \pi_i \}$ related with $\{ u_{ij} \}$ by means of (\ref{eq_lc_c}). Let us consider the fibre bundle $\wa \mF^\alpha \to X$ associated with $\mF^\alpha$; then we have an inclusion $\wa \mF^\alpha \subseteq \wa \mF$ and (\ref{eq_gact}) implies
\[
\pi_i( \wa \mF^\alpha |_{X_i} ) = X_i \times \mA_\bullet
\ \Rightarrow \
X_{ij} \times \mA_\bullet = \pi_i \circ \pi_j^{-1} ( X_{ij} \times \mA_\bullet )
\ .
\]
We conclude by (\ref{eq_lc_c}) that $v_{ij,x} :=$ $u_{ij,x} |_{\mA_\bullet} \in$ ${\bf aut}\mA_\bullet$ for every $x \in X_{ij}$ and pair $i,j$. Moreover, by (\ref{eq_QG_act}) we find that $v_{ij} = p \circ u_{ij}$, as $i,j$ vary, yield a set of transition maps for $\wa \mF^\alpha$. Finally, since $\mA$ is $QG$-$C_0(X)$-isomorphic to $\mF^\alpha$, we conclude that $[\efq] = p_*[\efn]$.

\noindent (2) $\Rightarrow$ (1): Let $\efn :=$ $\{ u_{ij} \}$. We define $\mF$ as the $\mF_\bullet$-bundle with cocycle $\efn$ and $\mG \to X$ as the fibre bundle with fibre $G$ and transition maps $\gamma_{ij,x}(g) :=$ $u_{ij,x} \circ g \circ u_{ij,x}^{-1}$, $x \in X_{ij}$. Such transition maps define a cocycle with class $\gamma_*[\efn] \in H^1(X,{\bf aut}G)$. Now, we note that
\[
u_{ij,x} \circ g (v_\bullet) 
\ = \ 
\gamma_{ij,x} (g) \circ u_{ij,x}(v_\bullet)
\ \ , \ \
g \in G \ , \ v_\bullet \in \mF_\bullet \ , \ x \in X_{ij}
\ .
\]
This implies that if we consider the maps 
\[
\alpha_i : 
(X_i \times G) \times_{X_i} (X_i \times \mF_\bullet) 
\ \to \ 
X_i \times \mF_\bullet
\ \ , \ \
\alpha_i \left( (x,g) , (x,v_\bullet) \right) := (  x , g(v_\bullet)  )
\ ,
\]
then there is a unique gauge action $\alpha : \mG \times_X \wa \mF  \to  \wa \mF$ with local charts $\{ \eta_i \}$ of $\mG$ associated with $\{ \gamma_{ij} \}$ and $\{ \pi_i \}$ of $\wa \mF$ associated with $\{ u_{ij} \}$, fulfilling 
\[
\alpha_i 
\ = \ 
\pi_i  \circ \alpha \circ \left( \eta_i^{-1} \times \pi_i^{-1} \right)
\]
for every index $i$. Since $\mA$ has $QG$-cocycle $\{ p \circ u_{ij} \}$, reasoning as in the first part of the proof we conclude that $\mA$ is $QG$-$C_0(X)$-isomorphic to $\mF^\alpha$.
\end{proof}

\begin{cor}
With the notation of the previous Lemma, if $\gamma_*[\efn] = [{\bf 1}]$ then there is a continuous action $\alpha_\bullet : G \to {\bf aut} \mF$ with fixed-point algebra $QG$-$C_0(X)$-isomorphic to $\mA$.
\end{cor}

\begin{proof}
Since $\gamma_*[\efn] = [{\bf 1}]$ there is an isomorphism $\mG \simeq X \times G$. Thus, the gauge action $\alpha : \mG \times_X \wa \mF$ induces the continuous action $\alpha_\bullet : G \to {\bf aut}_X \mF$ (see \cite[Cor.3.4]{Vas07}).
\end{proof}

\begin{thm}
\label{thm_dd2}
Let $G \subseteq K \subseteq {\bf aut}\mF_\bullet$, $\mA_\bullet$ denote the fixed-point algebra of $\mF_\bullet$ w.r.t. the $G$-action and $QG$ defined as in (\ref{def_gqg}(2)). For each $\mA_\bullet$-bundle $\mA$ with structure group $QG$ there is a class
\begin{equation}
\label{eq_dd2}
\delta (\mA) \in H^2(X,{G'})
\end{equation}
fulfilling the following property: if $\mA$ is $QG$-$C_0(X)$-isomorphic to the fixed-point algebra of a gauge \sC dynamical system $(\mF,\mG,\alpha)$ with fibre $(\mF_\bullet,G)$ and structure group $NG$, then $\delta (\mA) = [{\bf 1}]$. The converse is also true when $G$ lies in the centre of $NG$.
\end{thm}

\begin{proof}
Applying Lemma \ref{lem_ddc} we define $\delta(\mA) := \delta[\efq]$, where $\efq$ is the $QG$-cocycle associated with $\mA$ (of course, there is an abuse of the notation $\delta$ in the previous definition, but this should not create confusion). The theorem now follows applying Lemma \ref{lem_dd}.
\end{proof}

The class $\delta$ may be also interpreted as an obstruction to constructing covariant representations of a gauge \sC dynamical system over a continuous field of Hilbert spaces. Since this point goes beyond the purpose of the present work, we postpone a complete discussion to a forthcoming paper.

In the following lines we discuss the relation between the class $\delta$ and the Dixmier-Douady invariant. Let $\bH$ denote the standard separable Hilbert space, $\bU$ the unitary group of $\bH$ endowed with the norm topology, $\bT$ the torus acting on $\bH$ by scalar multiplication, $\bPU := \bU / \bT$ the projective unitary group, $\mK^r$ the \sC algebra of compact operators acting on the tensor power $\bH^r$, $r \in \bN$, and $(\bH^r,\bH^r) \supset \mK^r$ the \sC algebra of bounded operators. Moreover, let $\mO_\infty$ denote the Cuntz algebra; it is well-known that there is a continuous action
\begin{equation}
\label{eq_Uact}
\bU \to {\bf aut}\mO_\infty
\ ,
\end{equation}
defined as in (\ref{def_cag}), which restricts to the {\em circle action}
$\bT \to {\bf aut}\mO_\infty$.
The construction (\ref{def_gqg}) with $K = \bU$, $G = \bT$ yields $QG = \bPU$ and the action
\[
\bPU \to {\bf aut}\mO_\infty^0
\ \ , \ \
\gamma \mapsto \wa \gamma
\ .
\]
Now, $\mO_\infty^0$ can be constructed using a universal construction on $\mK$, as follows (see \cite{CP95}). Consider the inductive structure
\begin{equation}
\label{eq_i_l}
\ldots 
\ \stackrel{j_{r-1}}{\to} \ 
(\bH^r,\bH^r)
\ \stackrel{j_r}{\to} \ 
(\bH^{r+1},\bH^{r+1})
\ \stackrel{j_{r+1}}{\to} \ 
\ldots
\ \ , \ \
j_r(t) := t \otimes 1
\ ,
\end{equation}
where $1 \in (\bH,\bH)$ is the identity, and denote the associated \sC algebra by $\mB_\infty$. Then, $\mO_\infty^0$ is the \sC subalgebra of $\mB_\infty$ generated by the images of the $\mK^r \subset (\bH^r,\bH^r)$, $r \in \bN$. The $\bPU$-action on $\mO_\infty^0$ preserves the inductive structure: if $i_r : \mK^r \to \mO_\infty^0$, $r \in \bN$, are the natural inclusions, then
\begin{equation}
\label{eq_pu1}
\wa \gamma \circ i_r (t) \in i_r(\mK^r) 
\ \ , \ \ 
\gamma \in \bPU
\ , \ 
t \in i_r(\mK^r)
\ ,
\end{equation}
and in particular $\bPU$ acts on $i_1(\mK)$ as the usual adjoint action:
\begin{equation}
\label{eq_pu2}
\wa \gamma \circ i_1 (t) = i_1 \circ \gamma (t)
\ \ , \ \
t \in \mK
\ , \
\gamma \in \bPU
\ .
\end{equation}
Let us denote the category of $\mO_\infty^0$-bundles over $X$ with arrows $\bPU$-$C_0(X)$-isomorphisms by ${\bf bun}_\bPU(X,\mO_\infty^0)$. By the above results, each $\mO_\infty^0$-bundle $\mA_\infty$ with structure group $\bPU$ is determined by a $\bPU$-cocycle $\efq$, and the class
\begin{equation}
\label{eq_DD_O}
\delta (\mA_\infty) = \delta [\efq] \in H^2(X,\bT) \simeq H^3(X,\bZ)
\end{equation}
measures the obstruction to finding a gauge dynamical system with fibre (\ref{eq_Uact}) and fixed-point algebra $\mA_\infty$.
Now, we denote the category of $\mK$-bundles over $X$ with arrows $C_0(X)$-isomorphisms by ${\bf bun}(X,\mK)$; each $\mK$-bundle $\mA$ is determined by a $\bPU$-cocycle $\efq$, and its Dixmier-Douady invariant (\cite[Ch.10]{Dix}) is computed by (\ref{eq_dd_class}):
\begin{equation}
\label{eq_DD_K}
\delta_{DD}(\mA) = \delta (\efq) \ .
\end{equation}
%
%
%
%
\begin{prop}
\label{ex_dd}
For each locally compact, paracompact Hausdorff space $X$, there is an equivalence of categories
${\bf bun}(X,\mK) \to {\bf bun}_\bPU(X,\mO_\infty^0)$,
$\mA \mapsto \mA_\infty$, and
\begin{equation}
\label{eq_DD_O_K}
\delta_{DD}(\mA) = \delta(\mA_\infty)
\ \ , \ \
\mA \in {\bf bun}(X,\mK) \ .
\end{equation}
\end{prop}

\begin{proof}
Let $\mA$ be a $\mK$-bundle with associated $\bPU$-cocycle $\efq$. The multiplier algebra $M \mA$ of $\mA$ can be constructed as the $C_0(X)$-algebra of bounded sections of the bundle $\wa \mB \to X$ with fibre $(\bH,\bH)$ and transition maps defined by $\efq$. For each $r \in \bN$, we consider the $C_0(X)$-tensor products $M \mA^r := M \mA \otimes_X \ldots \otimes_X M \mA$, $\mA^r := \mA \otimes_X \ldots \otimes_X \mA$ and the obvious inclusions $\mA^r \subset M \mA^r$. We have the inductive limit structure
\[
\ldots 
\ \stackrel{j_{r-1}}{\to} \ 
M \mA^r 
\ \stackrel{j_r}{\to} \ 
M \mA^{r+1} 
\ \stackrel{j_{r+1}}{\to} \ 
\ldots
\ \ , \ \
j_r(t) := t \otimes 1
\ ,
\]
where $1 \in M \mA$ is the identity. The system $( M \mA^r , j_r )$ yields the inductive limit algebra $M \mA_\infty$ and we define $\mA_\infty$ as the \sC subalgebra of $M \mA_\infty$ generated by the images of the \sC algebras $\mA^r$, $r \in \bN$. 
If $\beta : \mA \to \mA'$ is a $C_0(X)$-isomorphism, then it naturally extends to $C_0(X)$-isomorphisms $\beta_r : \mA^r \to {\mA'}^r$, $r \in \bN$, and finally to a $C_0(X)$-isomorphism $\beta_\infty : \mA_\infty \to \mA'_\infty$.
On the converse, let $\mA_\infty$ be a $\mO_\infty^0$-bundle with structure group $\bPU$ and associated $\bPU$-cocycle $\efq$. Since the $\bPU$-action on $\mO_\infty^0$ preserves the inductive structure (\ref{eq_i_l}), and since the $\bPU$-action on $\mO_\infty^0$ restricts to the natural $\bPU$-action on $\mK \subset \mO_\infty^0$ (see (\ref{eq_pu1}) and (\ref{eq_pu2})), for each $r \in \bN$ there is a $\mK^r$-bundle $\mA^r \subset \mA_\infty$ with associated $\bPU$-cocycle $\efq$, with $\mA^1$ generating $\mA_\infty$ as above; thus our functor is surjective on the sets of objects. 
If $\beta' : \mA_\infty \to \mA'_\infty$ is an isomorphism in ${\bf bun}_\bPU(X,\mO_\infty^0)$ then by $\bPU$-equivariance we find $\beta' |_{\mA^r} = {\mA'}^r$ for each $r \in \bN$. Defining $\beta := \beta' |_{\mA^1}$ we easily find $\beta' = \beta_\infty$; thus our functor is surjective on the sets of arrows. 
Finally, (\ref{eq_DD_O_K}) follows by (\ref{eq_DD_O}) and (\ref{eq_DD_K}). 
\end{proof}

\section{Gauge-equivariant bundles, and a concrete duality.}
\label{sec_gd}

Let $X$ be a compact Hausdorff space. In the present section we give a duality theory in the setting of the category ${\bf vect}(X)$ of vector bundles over $X$, relating suitable subcategories of ${\bf vect}(X)$ with gauge equivariant vector bundles in the sense of \cite{NT04}.

Let $d \in \bN$ and $\pi :\mE \to X$ a vector bundle of rank $d$. We denote the Hilbert $C(X)$-bimodule of sections of $\mE$ by $\wE$, endowed with coinciding left and right $C(X)$-actions. For each $r \in \bN$, we denote the $r$-fold tensor power of $\mE$ in the sense of \cite[\S I.4]{Kar}, \cite[1.2]{Ati} by $\mE^r$ (for $r = 0$, we define $\mE^0 := \iota := X \times \bC$) and by $(\ers)$ the set of vector bundle morphisms from $\mE^r$ into $\mE^s$. 
The Serre-Swan equivalence implies that every $(\ers)$ is the 
$C(X)$-bimodule of sections of a vector bundle $\pi_{rs} : \mE^{rs} \to X$, having fibre $(\hrs) \equiv  \bM_{d^r,d^s}$ (\cite[Thm.5.9]{Kar}). In explicit terms, $\mE^{rs} \simeq \mE^s \otimes \mE_*^r$, where $\mE_*^r$ is the $r$-fold tensor power of the conjugate bundle and every $t \in (\ers)$ can be regarded as a continuous map
\[
t : X \to \mE^{rs} 
\ \ , \ \
\pi_{rs} \circ t = id_X 
\ .
\]
We denote the tensor category with objects $\mE^r$, $r \in \bN$, and arrows $(\ers)$ by $\wa \mE$. It is clear that $(\ii) = C(X)$. Moreover, the flip operator
\begin{equation}
\label{eq_flip}
\theta_\mE  \in ( \mE^2 , \mE^2 ) 
\ : \ 
\theta_\mE (x) \circ ( v \otimes v') := v' \otimes v
\ \ , \ \
v,v' \in \mE_x
\ \ , \ \
x \in X
\ ,
\end{equation}
defines a symmetry on $\wa \mE$. Thus, $( \wa \mE , \otimes , \iota , \theta_\mE  )$ is a symmetric tensor \sC category; we denote the associated pointed \sC dynamical system by $( \coe , \sigma_\mE , \theta_\mE )$.
\begin{prop}
\label{prop_we}
Let $d \in \bN$ and $\bH$ denote the standard rank $d$ Hilbert space. (1) For each compact Hausdorff space $X$ there is an isomorphism 
\[
Q : \sym ( X,\wa \bH ) \to H^1(X,\ud)
\ .
\]
(2) If $\mE \to X$ is a rank $d$ vector bundle, then the category $( \wa \mE , \otimes , \iota , \theta_\mE  )$ is a $\wa \bH$-bundle and all the elements of $\sym (X,\wa \bH)$ are of this type; (3) If $\efu$ is an $\ud$-cocycle associated with $\mE$ as a set of transition maps, then $Q[\wa \mE] = [\efu]$; (4) $\coe$ is the Cuntz-Pimsner algebra associated with $\wE$ and is an $\mO_d$-bundle with structure group $\ud$.
\end{prop}

\begin{proof}
(1) We apply Thm.\ref{thm_amen2} to the case $G = \{ 1 \}$, so that $NG = QG = \ud$. 
(2) Let $\mE \to X$ be a vector bundle; we consider a local chart $\pi_U : \mE |_U \to U \times H$ and note that, by functoriality, for each $r,s \in \bN$ there are local charts $\pi_U^{rs} : \mE^{rs} |_U \to X \times (\hrs)$. This yields the desired local chart $\wa \pi_U : \wa \mE \to U \wa \bH$.
Let now $( \wa \rho , \otimes , \iota , \eps )$ be a $\wa \bH$-bundle; to prove that $\wa \rho \simeq \wa \mE$ for some vector bundle $\mE$ we note that the Hilbert $C(X)$-bimodule $( \iota , \rho )$ defines a locally trivial continuous field of Hilbert spaces with fibre $\bH$; we denote the vector bundle associated with $( \iota , \rho )$ by $\mE$, and applying the Serre-Swan equivalence we obtain an isomorphism $\beta : \wE \simeq ( \iota , \mE ) \to (\iota , \rho)$, which extends to the desired isomorphisms $\beta^{rs} : (\ers) \to (\rhors)$, $r,s \in \bN$. 
(3) We pick an $\ud$-cocycle $\efu'$ with class $Q[\wa \mE]$. By definition of $Q$ we have that $\efu'$ yields transition maps for the vector bundles $\mE^{rs}$, $r,s \in \bN$, by means of the action $\wa u (t) :=$ $u_s \circ t \circ u_r^*$, $u \in \ud$, $t \in (\hrs)$ (compare with (\ref{eq_ga_rs})). In particular, for $r=0$, $s=1$, we conclude that $\efu'$ defines, up to cocycle equivalence, a set of transition maps for $\mE$ and thus $[\efu'] = Q[\wa \mE] = [\efu]$.
(4) It suffices to recall \cite[Prop.4.1, Prop.4.2]{Vas}.
\end{proof}

\begin{rem}
\label{rem_vb_lc}
To be concise, we denote the totally antisymmetric projections defined as in (\ref{def_ap}) by $P_n :=$ $P_{\mE , \theta_\mE , n} \in$ $( \mE^n , \mE^n )$, $n \in \bN$. By definition of the totally antisymmetric line bundle $\wedge^d \mE := P_d \mE^d$ we have that $\mE$ is a twisted special object, with 'categorical Chern class' (\ref{def_cc_tso}) coinciding with the first Chern class $c_1(\mE)$. If $c_1(\mE) = 0$ then $\mE$ is a special object and the conjugate bundle $\mE_*$ appears as the object associated with the projection $P_{d-1} \in ( \mE^{d-1} , \mE^{d-1} )$ (see \cite[Lemma 3.6]{DR89}). Clearly, the existence of the conjugate bundle does not depend on the vanishing of $c_1(\mE)$, anyway in general it is false that $\mE_* \simeq P_{d-1} \mE^{d-1}$.
\end{rem}

Let $\wa \rho$ be a tensor $C(X)$-subcategory of $( \wa \mE , \otimes , \iota )$; we denote the spaces of arrows of $\wa \rho$ by $(\ers)_\rho$, $r,s \in \bN$. For every $r,s \in \bN$, we define the set $\mE^{rs}_\rho :=$ $\left\{  t_x  \in \mE^{rs} : \right.$ $x \in X ,$ $\left. t \in (\ers)_\rho \right\}$ and denote the restriction of $\pi_{rs}$ on $\mE^{rs}_\rho$ by $\pi_{rs}^\rho$. In this way, we obtain Banach bundles
\begin{equation}
\label{def_ters}
\pi_{rs}^\rho : \mE^{rs}_\rho \to X \ .
\end{equation}
Let $t \in (\ers)$. If $t \in (\ers)_\rho$, then by definition  $t_x \in \mE^{rs}_\rho$ for every $x \in X$. On the converse, suppose that $t_x \in \mE^{rs}_\rho$, $x \in X$; then for every $x \in X$ there is $t' \in (\ers)_\rho$ such that $t_x = t'_x$. By continuity, for every $\eps > 0$ there is a neighbourhood $U_\eps \ni x$ with $\sup_{y \in U_\eps} \left\| t_y - t'_y \right\| < \eps$. Thus, \cite[10.1.2 (iv)]{Dix} implies that $t \in (\ers)_\rho$. We conclude that 
\begin{equation}
\label{eq_t_ersrho}
t \in (\ers)_\rho
\ \Leftrightarrow \
t_x \in \mE^{rs}_\rho
\ , \ 
\forall x \in X \ .
\end{equation}
Let $p : \mG \to X$ be a group bundle with fibres compact groups $G_x := p^{-1}(x)$, $x \in X$. According to \cite{NT04}, a {\em gauge action} on $\mE$ is given by a continuous map 
\[
\alpha : \mG \times_X \mE \to \mE \ ,
\]
such that each restriction $\alpha_x : G_x \times \mE_x \to \mE_x$, $x \in X$, is a unitary representation on the Hilbert space $\mE_x$; to economize on notation, we define
\[
G_{\alpha,x} := \alpha_x (G_x)
\ \ , \ \
u_{\alpha,x} := \alpha_x (u)
\ \ , \ \
u \in G_x
\ .
\]
In this way, $\mE$ is a $\mG${\em-equivariant vector bundle} in the sense of \cite[\S 1]{NT04}, with trivial action on $X$. Moreover, every $\pi_{rs} :  \mE^{rs} \to X$ is a $\mG$-vector bundle, with action
\[
\alpha^{rs} : \mG \times_X \mE^{rs} \to \mE^{rs} 
\ , \ 
(u,v) \mapsto \alpha^{rs} (u,v) := \wa u_{\alpha,x} (v)
\ , \ 
x := p (u) = \pi_{rs}(v) \in X  \  ,
\]
where $\wa u_{\alpha,x} (v)$ is defined as in (\ref{def_cag}). We denote the category with objects $\mE^r$, $r \in \bN$, and arrows
\begin{equation}
\label{def_dual}
(\ers)_\alpha
:= 
\left\{ 
t \in ( \ers ) 
\ : \ 
\alpha^{rs} ( u, t(x) ) = t(x) 
\ , \ 
u \in \mG , x := p(u)
\right\}
\end{equation}
by $\wa \alpha$.
Clearly, $( \wa \alpha , \otimes , \iota )$ is a tensor \sC category with $(\ii) = C(X)$ and fibres $\wa G_{\alpha,x}$, $x \in X$, defined as in (\ref{eq_ga_rs}). Since $\theta_\mE (x) = \theta$, $x \in X$, we conclude that $\theta_\mE \in (\mE^2,\mE^2)_\alpha$, thus there is an inclusion functor
\[
E : 
( \wa \alpha , \otimes , \iota , \theta_\mE )
\to
( \wa \mE , \otimes , \iota , \theta_\mE )
\ .
\]
Let us consider the bundle $\mcUE \to X$ of unitary automorphisms of $\mE$ (see \cite[I.4.8]{Kar}). It is well known that $\mcUE$ has fibre the unitary group $\ud$; if $\{ u_{ij} \}$ is the $\ud$-cocycle associated with $\mE$, then $\mcUE$ has associated ${\bf aut} \ud$-cocycle
\[
\gamma_{ij,x} (u) := u_{ij,x} \cdot u \cdot u_{ij,x}^*
\ \ , \ \
x \in X_{ij}
\ , \
u \in \ud
\ .
\]
Note that $\mcUE$ is compact as a topological space. In the same way the bundle $\mcSUE \to X$ of special unitary automorphisms of $\mE$ is defined: it has fibre $\sud$ and the same transition maps as $\mcUE$. Of course, there is an inclusion $\mcSUE \subset \mcUE$.

Now let be $\mG \to X$ be a closed subbundle of $\mcUE$, not necessarily locally trivial. Then there is an obvious gauge action $\alpha : \mG \times_X \mE \to \mE$. In order to emphasize the picture of $\mG$ as a subbundle of $\mcUE$, we use the notations 
\[
\wa \mG := \wa \alpha
\ \ , \ \
(\ers)_\mG := (\ers)_\alpha
\ ,
\]
and call $\wa \mG$ the {\em dual} of $\mG$. Clearly, each $(\ers)_\mG$ is the module of sections of a Banach bundle
\[
\pi_{rs}^\mG : \mE^{rs}_\mG \to X
\ \ , \ \
r,s \in \bN
\ .
\]
We define $( \mO_\mG , \sigma_\mG , \theta_\mE )$ as the pointed \sC dynamical system associated with $( \wa \mG , \otimes , \iota , \theta_\mE )$. Clearly, there is a canonical monomorphism
\[
E_* : ( \mO_\mG , \sigma_\mG , \theta_\mE ) \to ( \coe , \sigma_\mE , \theta_\mE ) \ .
\]
Actions on the vector bundle $\mE \to X$ by (generally noncompact) groups $G$ of unitary automorphisms have been considered in \cite[\S 4]{Vas05}. This approach has the disadvantage to associate the same dual to very different groups (see \cite[Ex.4.2]{Vas05}). According to \cite[Def.4.7]{Vas05}, we can associate a group bundle $\mG \subseteq \mcUE$ to $G$, in such a way that the map $\{ \mG \mapsto \wa \mG \}$ is one-to-one (\cite[Prop.4.8]{Vas05}). For this reason in the present paper we passed to consider the notion of gauge action.

The following result is a different version of \cite[Prop.4.8]{Vas05}; since the proof is essentially the same, it is omitted.
\begin{prop}
\label{rem_dual}
Let $\mE \to X$ be a vector bundle. The map $\{ \mG \mapsto \wa \mG \}$ defines a one-to-one correspondence between the set of closed subbundles of $\mcSUE$ and the set of symmetric tensor \sC subcategories $\wa \rho$ of $\wa \mE$ such that $( \iota , \wedge^d \mE ) \subseteq (\iota , \mE^d)_\rho$.
\end{prop}

%
%

Let $G \subseteq \ud$. A {\em $\wa G$-bundle in} $\wa \mE$ is a $\wa G$-bundle $\wa \rho$ endowed with an inclusion
\[
( \wa \rho , \otimes , \iota , \theta_\mE )
\ \subseteq \
( \wa \mE , \otimes , \iota , \theta_\mE )
\ .
\]
Let us denote the inclusion map by $i : NG \to \ud$, and the quotient projection by $p : NG \to QG$; by functoriality of $H^1(X,\cdot \ )$, there are maps
\begin{equation}
\label{eq_p*}
i_* : H^1(X,NG) \to H^1(X,\ud)
\ \ , \ \
p_* : H^1(X,NG) \to H^1(X,QG)
\ .
\end{equation}
Moreover, by (\ref{eq_gamma}) each $NG$-cocycle $\efn = \{ u_{ij} \}$ defines an ${\bf aut}G$-cocycle $\{ \wa u_{ij} \}$ with class $\gamma_*[\efn]$.

\begin{thm}
\label{thm1_str_gr}
Let $G \subseteq \ud$ be a compact group. For each compact Hausdorff space $X$ and $NG$-cocycle $\efn = \{ u_{ij} \}$, there are a vector bundle $\mE \to X$ with $\ud$-cocycle $\{ i \circ u_{ij} \}$ and a fibre $G$-bundle $\mG \subseteq \mcUE$ with transition maps $\{ \wa u_{ij} \}$. The category $( \wa \mG , \otimes , \iota , \theta_\mE )$ is a $\wa G$-bundle with associated cohomology class $p_*[\efn] \in H^1(X,QG)$. Moreover, there is a gauge action
\[
\alpha : \mG \times_X \wa \mO_\mE \to \wa \mO_\mE
\]
with fibre $( \mO_d , G )$ and fixed-point algebra $\mO_\mG$.
\end{thm}

\begin{proof}
Clearly, there are $\mE$ and $\mG$ defined as above. Since by construction $\mG \subseteq \mcUE$, the action $\alpha$ is defined, together with the tensor \sC category $( \wa \mG , \otimes , \iota , \theta_\mE )$ and the pointed \sC dynamical system $( \mO_\mG , \sigma_\mG , \theta_\mE )$.
By (\ref{eq_qg}) we can regard $QG$ as a subgroup of ${\bf aut}\wa G$, so the $QG$-cocycle $\efq := \{ p \circ u_{ij} \}$ defines a symmetric tensor \sC category $( \wa \rho_\efq , \otimes , \iota , \eps_\efq )$ and a pointed \sC dynamical system $( \mO_\efq , \rho_\efq , \eps_\efq )$.
To prove that $\wa \mG$ has associated cocycle $\efq$, it suffices to give a $QG$-$C(X)$-isomorphism $\mO_\efq \simeq \mO_\mG$.
To this end, we note that Lemma \ref{lem_dd} implies that $\mO_\efq$ is $QG$-$C(X)$-isomorphic to the fixed-point algebra $\mO_\mE^\alpha$; thus, in order to get the desired isomorphism, it suffices to prove that $\mO_\mG = \mO_\mE^\alpha$. Now, it is clear that $\mO_\mG \subseteq \mO_\mE^\alpha$. To prove the opposite inclusion, we consider the Haar functional $\varphi : C(\mG) \to C(X)$ and the induced invariant mean $m : \coe \to \mO_\mE^\alpha$ in the sense of \cite[\S 4]{Vas07}. By definition of $\alpha$ we have $m((\ers)) = (\ers)_\mG \subset \mO_\mG$, $r,s \in \bN$, so that if $t \in \mO_\mE^\alpha$ is a norm limit of the type $t = \lim_n t_n$, $t_n \in {\mathrm{span}} \cup_{rs} (\ers)$, then $t = m(t) = \sum_n m(t_n)$, with $t_n \in (\ers)_\mG$. Thus, $\mO_\mE^\alpha = \mO_\mG$ and this completes the proof.
\end{proof}

In the following theorem we characterize the $\wa G$-bundles in $\wa \mE$ that arise as above.
\begin{thm}
\label{thm_str_gr}
Let $G \subseteq \ud$ be covariant, $\mE \to X$ a vector bundle with $\ud$-cocycle $\efu$ and $\wa \rho$ a $\wa G$-bundle in $\wa \mE$ with $QG$-cocycle $\efq$ (in the sense of Lemma \ref{lem_qg}). Then the structure group of $\mE$ can be reduced to $NG$, i.e. there is an $NG$-cocycle $\efn$ such that
$[\efu] = i_*[\efn]$.
Moreover $[\efq] = p_*[\efn]$ and $\wa \rho = \wa \mG$, where $\mG \subseteq \mcUE$ is a fibre $G$-bundle with class $\gamma_*[\efn] \in H^1(X,{\bf aut}G)$.
\end{thm} 

\begin{proof}
We associate to $\wa \rho$ the locally trivial pointed \sC dynamical system $( \oro , \rho_* , \theta_\mE )$ with fibre $( \mO_G , \sigma_G , \theta )$, equipped with the inclusion $( \oro , \rho_* , \theta_\mE ) \subseteq ( \coe , \sigma_\mE , \theta_\mE )$. There is a finite open cover $\{ X_i \}$ and local charts $\eta_i : \oro |_{X_i} \to C_0(X_i) \otimes \mO_G$, defining the $QG$-cocycle $\efq :=$ $\{ y_{ij} \}$ such that
\[
\wa y_{ij,x}
\ = \
\eta_{i,x} \circ \eta_{j,x}^{-1}
\ \ , \ \
x \in X_{ij}
\ .
\]
Now, up to performing a refinement, we may assume that $\left\{ X_i \right\}$ trivializes $\mE$, so that there are local charts $\pi_i : \mE |_{X_i} \to X_i \times \bH$ with associated $\ud$-cocycle $\efu := \left\{ u_{ij} := \pi_i \circ \pi_j^{-1} \right\}$. Moreover, each $\pi_i$ induces a local chart $\wa \pi_i : \coe \to C_0(X_i) \otimes \mO_d$.
Let us define $\mO_{\rho,i} := \wa \pi_i (\oro)$. We introduce the $C_0(X_i)$-isomorphisms 
\[
\beta_i := \wa \pi_i \circ \eta_i^{-1}
\ \ , \ \ 
\beta_i : C_0(X_i) \otimes \mO_G \stackrel{\simeq}{\lora} \mO_{\rho,i} \subseteq C_0(X_i) \otimes \mO_d
\ ,
\]
so that for each pair $i,j$ we find
\begin{equation}
\label{eq_coc}
\wa y_{ij,x} = \beta_{i,x}^{-1} \circ \wa u_{ij,x} \circ \beta_{j,x} 
\ \ , \ \
x \in X_{ij}
\ .
\end{equation}
Now, each $\beta_i$ may be regarded as a continuous map $\beta_i : X_i \to {\bf emb}\mO_G$, thus by Lemma \ref{cor_dual_od} there is an open cover $\{ Y_{il} \}_l$ of $X_i$ and continuous maps $w_{il} : Y_{il} \to \ud$ such that 
\begin{equation}
\label{eq_coc1}
\beta_{i,x}(t) = \wa w_{il,x}(t) \ \ , \ \ t \in \mO_G \ , \ x \in Y_{il} \ .
\end{equation}
We extract from $\{ Y_{il} \}_{il}$ a finite open cover $\{ Y_h \}$ of $X$; to economize on notation, we introduced the index $h$ instead of $i,l$, so that we have maps $w_h$ satisfying (\ref{eq_coc1}) for each $h$ and $x \in Y_h$. Since $\mE |_{Y_h}$ is trivial, we have that $\efu$ is equivalent to a cocycle defined by transition maps $u_{hk} : Y_{hk} \to \ud$; with a slight abuse of notation, we denote this cocycle again by $\efu$. Of course, the same procedure applies to $\efq = \{ y_{hk} \}$. By (\ref{eq_coc}), we find
\begin{equation}
\label{eq_coc_2}
\wa y_{hk,x} = \wa w_{h,x} \circ \wa u_{hk,x} \circ \wa w_{k,x}^{-1} 
\ \ , \ \
x \in Y_{hk}
\ .
\end{equation}
Now, $\efu$ is equivalent to the $\ud$-cocycle $\efn :=$ $\{ z_{hk} := w_h u_{hk} w^*_k \}$ and (\ref{eq_coc_2}) becomes 
\[
\wa y_{hk,x} (t) = \wa z_{hk,x} (t)
\ \ , \ \
x \in Y_{hk}
\ , \
t \in \mO_G 
\ .
\]
\noindent In other terms, each $z_{hk,x} \in {\bf aut}(\mO_d,\sigma_d,\theta)$ restricts to the automorphism $\wa y_{hk,x} \in {\bf aut}\mO_G$, so by Rem.\ref{rem_amen2} we conclude that $\efn$ takes values in $NG$ and yields a reduction to $NG$ of the structure group of $\mE$. 
Moreover, using again Rem.\ref{rem_amen2} we have $[\efq] = p_*[\efn]$. The fact that $\mG$ has class $\gamma_*[\efn]$ follows by applying the previous theorem to the $NG$-cocycle $\efn$.
\end{proof}

\begin{cor}
\label{cor_str_gr}
Let $\wa \rho$ be a special category with an inclusion $( \wa \rho , \otimes , \iota , \theta_\mE ) \to ( \wa \mE , \otimes , \iota , \theta_\mE )$. Then: (1) There is a compact group $G \subseteq \sud$ such that $\mE$ has an associated $NG$-cocycle $\efn$; (2) $\wa \rho$ has class $Q[\wa \rho] = p_*[\efn] \in H^1(X,QG)$; (3) There is a group bundle $\mG \subseteq \mcSUE$ such that $\wa \rho = \wa \mG$.
\end{cor}

\begin{proof}
(1) By Thm.\ref{thm_amen} there is a compact group $G \subseteq \sud$ such that $\wa \rho$ has fibre $\wa G$ and an associated $QG$-cocycle $\efq$; thus, by the previous theorem we conclude that $\mE$ has an associated $NG$-cocycle $\efn$. (2) The previous theorem implies $[\efq] = p_*[\efn]$. (3) We apply again the previous theorem.
\end{proof}

\section{Cohomological invariants and duality breaking.}
\label{class}

In the present section we approach the following question: given a covariant inclusion $G \subseteq \ud$ and a $\wa G$-bundle $( \wa \rho , \otimes , \iota , \eps)$, is there any $\mG$-equivariant vector bundle $\mE \to X^\iota$ with an isomorphism $( \wa \mG , \otimes , \iota , \theta_\mE ) \simeq$ $( \wa \rho , \otimes , \iota , \eps)$? This is what we call the problem of {\em abstract} duality, as -- differently from the previous section -- our category $\wa \rho$ is not presented as a subcategory of ${\bf vect}(X^\iota)$.
We will give a complete answer to the previous question in terms of the cohomology set $H^1(X^\iota,QG)$, reducing the problem of abstract duality to (relatively) simple computations involving cocycles and principal bundles.

As a preliminary step we analyze the setting of \sC bundles. Let $X$ be a compact Hausdorff space. By Lemma \ref{lem_qg} we have that $H^1(X,QG)$ describes the set of isomorphism classes of locally trivial, pointed \sC dynamical systems with fibre $( \mO_G , \sigma_G , \theta )$. For every $QG$-cocycle $\efq$, we denote the associated pointed \sC dynamical system by $( \mO_\efq , \rho_\efq , \eps_\efq )$.
\begin{thm}
\label{thm_emb_O}
With the above notation, for each $QG$-cocycle $\efq$ and $( \mO_\efq , \rho_\efq , \eps_\efq )$, the following are equivalent:
\begin{enumerate}
\item There is a rank $d$ vector bundle $\mE \to X$ with a $C_0(X)$-monomorphism $\eta : ( \mO_\efq , \rho_\efq , \eps_\efq ) \to ( \coe , \sigma_\mE , \theta_\mE )$;
\item There is a gauge \sC dynamical system $( \mO , \mG , \alpha )$ with fibre $( \mO_d , G )$ and structure group $\ud$, such that $\mO_\efq$ is $QG$-$C_0(X)$-isomorphic to the fixed-point algebra $\mO^\alpha$;
\item There is an $NG$-cocycle $\efn$ such that $p_*[\efn] = [\efq]$.
\end{enumerate}
\end{thm}

\begin{proof}
(3) $\Rightarrow$ (2): we consider the Cuntz algebra $\mO_d$ endowed with the $NG$-action (\ref{def_cag}), which factorizes through the action $QG \to {\bf aut}\mO_G$. Then we apply Lemma \ref{lem_dd} with $\mF_\bullet = \mO_d$ and $\mA_\bullet = \mO_G$. 
(2) $\Rightarrow$ (1): Let $\efn := \{ u_{ij} \}$ denote the $\ud$-cocycle associated with $\mO$ and $\mE \to X$ be the rank $d$ vector bundle with transition maps $\{ u_{ij} \}$. According to Prop.\ref{prop_we} there is a $\ud$-$C_0(X)$-isomorphism $\mO \simeq \coe$, so that, to be concise, we identify $\mO$ with $\coe$. Now, by construction of $\coe$ there is an inclusion $\wE \subset \coe$, to which corresponds an inclusion $\mE \subset \wa \mO_\mE$; since $( \coe , \mG , \alpha )$ has fibre $( \mO_d , G)$, with $G$ acting on $\mO_d$ as in (\ref{def_cag}), we conclude that $\mE$ is $\mG$-stable and the map $\alpha : \mG \times_X \wa \mO_\mE \to \wa \mO_\mE$ restricts to an action
\begin{equation}
\label{eq_emb_O}
\mG \times_X \mE \to \mE \ ,
\end{equation}
i.e., $\mE$ is $\mG$-equivariant. By Thm.\ref{thm1_str_gr}, we conclude that $\mO^\alpha = \mO_\mG$. Moreover, by Thm.\ref{thm_amen2} we have $QG = {\bf aut}( \mO_G , \sigma_G , \theta )$, thus, from Rem.\ref{rem_lt_pds} we conclude that the given $QG$-$C_0(X)$-isomorphism $\beta : \mO_\efq \to \mO_\mG$ yields monomorphisms
\begin{equation}
\label{eq_emb_O2}
( \mO_\efq , \rho_\efq   , \eps_\efq   ) \stackrel{\beta}{\lora}
( \mO_\mG , \sigma_\mG , \theta_\mE ) \hra
( \coe    , \sigma_\mE , \theta_\mE ) \ .
\end{equation}
(1) $\Rightarrow$ (3): Apply again Lemma \ref{lem_dd} with $\mF_\bullet = \mO_d$ and $\mA_\bullet = \mO_G$.
%
%
%
%
\end{proof}

Now, by Thm.\ref{thm_amen2} there are maps
\begin{equation}
\label{thm_class}
\left\{
\begin{array}{ll}
\sym ( X , \wa G ) \to H^1 ( X , QG ) 
\ \ , \ \ 
[ \wa \rho , \otimes , \iota , \eps ]  \mapsto  Q[\wa \rho]
\\
H^1 ( X,QG ) \to \sym ( X , \wa G )
\ \ , \ \
[\efq]  \mapsto [ \wa \rho_\efq , \otimes , \iota , \eps_\efq ]
\end{array}
\right.
\end{equation}
which are the inverses one of each other. The following result is the translation of Thm.\ref{thm_emb_O} in categorical terms; the proof is an immediate application of Thm.\ref{thm_oro}, Thm.\ref{thm_amen2} and Thm.\ref{thm_str_gr}, thus it will be omitted.
\begin{thm}
\label{thm_emb}
Let $d \in \bN$ and $G \subseteq \ud$ be covariant. For each $\wa G$-bundle $(\wa \rho , \otimes , \iota , \eps)$, the following are equivalent:
\begin{enumerate}
\item  There is an embedding functor $E : \wa \rho \hra {\bf vect}(X^\iota)$;
\item  There is a vector bundle $\mE \to X^\iota$ and a compact $G$-bundle $\mG \subseteq \mcUE$ with an isomorphism
$( \wa \rho , \otimes , \iota , \eps ) 
 \simeq
 ( \wa \mG ,\otimes , \iota , \theta_\mE )$;
\item There is an $NG$-cocycle $\efn$ such that $p_*[\efn] = Q[\wa \rho]$.
\end{enumerate}
\end{thm}

We call {\em a gauge group associated with} $\wa \rho$ the bundle $\mG \to X$ appearing in Thm.\ref{thm_emb}, whose isomorphism class is labeled by $\gamma_*[\efn] \in H^1(X,{\bf aut}G)$.
It follows from the previous theorem that the set of embedding functors $E : \wa \rho \hra {\bf vect}(X)$ is in one-to-one correspondence with the set 
of $NG$-cocycles $\efn$ such that $p_*[\efn] = Q[\wa \rho]$,
that we denote by $Z^1(X,NG;\wa \rho)$.
As we shall see in the sequel, $Z^1(X,NG;\wa \rho)$ may contain more than a cohomology class, or be empty. 
Let $\efn , \efn' \in Z^1(X,NG;\wa \rho)$ and $\mG$, $\mG'$ denote the associated gauge groups. In general, $\mG$ may be not isomorphic to $\mG'$; an example of this phenomenon with $G = {\mathbb{SU}}(2)$ is provided in Ex.\ref{ex_su2}.

\begin{cor}
\label{cor_cs}
Let $G \subseteq \ud$ be covariant. If there is a continuous monomorphism $s : QG \to NG$, $p \circ s = id_{QG}$, then for each $\wa G$-bundle $( \wa \rho , \otimes , \iota , \eps )$ there is at least one embedding functor $\wa \rho \hra {\bf vect}(X^\iota)$.
\end{cor}

\begin{proof}
By functoriality there is $s_* : H^1(X^\iota,QG) \to H^1(X^\iota,NG)$ such that $p_* \circ s_*$ is the identity on $H^1(X^\iota,QG)$. Thus $Q[\wa \rho] = p_*[\efn]$,  $\efn := s_* \circ Q[\wa \rho]$ and this means that the desired embedding functor exists.
\end{proof}

\begin{ex}
\label{ex_sud}
Let $G = \sud$, so that $NG = \ud$ and $QG = \bT$. By (\ref{thm_class}), we have the map
\[
Q : \sym (X,\wa \sud) \to H^1(X,\bT) \ .
\]
%
%
Elementary computations show that the quotient map $p : \ud \to \bT$ is the determinant; we define the continuous section $s : \bT \hra \ud$, $s (z) = z \oplus 1_{d-1}$, where $1_{d-1}$ is the identity of $\bM_{d-1}$. Since $s$ is multiplicative, we conclude by Cor.\ref{cor_cs} that for each $\wa \sud$-bundle $\wa \sigma$ there is at least one embedding functor 
$E : \wa \sigma \to {\bf vect} (X)$.
\end{ex}

For future reference, we consider the well-known isomorphism $B : H^1(X,\bT) \to H^2(X,\bZ)$. Moreover, we recall the reader to (\ref{def_cc_tso}).
\begin{cor}
\label{cor_dc}
Let $( \wa \rho , \otimes , \iota , \eps )$ be a special category such that $\rho$ has dimension $d \in \bN$ and Chern class $c \in H^2(X^\iota , \bZ)$. Then there is an $\wa \sud$-bundle $(\wa \sigma,\otimes,\iota,\eps)$ with an inclusion functor 
\begin{equation}
\label{eq_emb_rdc}
( \wa \sigma , \otimes , \iota , \eps ) 
\to 
( \wa \rho , \otimes , \iota , \eps )
\ .
\end{equation}
If $E : \wa \rho \hra {\bf vect}(X^\iota)$ is an embedding functor and $\mE := E(\rho)$, then there is a factorization
\begin{equation}
\label{eq_fact}
\xymatrix{
   \wa \sigma
   \ar[r]
   \ar[d]^-{E}_-{\simeq}
 & \wa \rho
   \ar[d]^-{E}
\\ \wa \mcSUE
   \ar[r]^-{\subseteq}
 & \wa \mE
}
\end{equation}
and $\mE$ has first Chern class $c_1(\mE) = c$.
\end{cor}

\begin{proof}
We define $\wa \sigma$ as the tensor \sC subcategory of $\wa \rho$ generated by the symmetry operators $\eps_\rho (r,s)$, $r,s \in \bN$, and the elements of $\mR_\rho$ (see (\ref{def_cc_tso}) and following remarks). The obvious inclusion $\wa \sigma \subseteq \wa \rho$ yields the functor (\ref{eq_emb_rdc}). 
If $E$ is an embedding functor then $E(P_{\rho,\eps,d}) = P_d$ (see Rem.\ref{rem_vb_lc}); this implies $E(\mR_\rho) = (\iota , \wedge^d \mE)$ and we conclude that $c(\rho) = c_1(\mE)$. Finally, since the spaces of arrows of $\wa \mcSUE$ are generated by the flips $\theta_\mE(r,s) = E(\eps_\rho (r,s))$, $r,s \in \bN$, and elements of $(\iota,\wedge^d\mE)$, we obtain the desired factorization (\ref{eq_fact}).
\end{proof}

\begin{cor}
Let $(\wa \sigma,\otimes,\eps,\iota)$ be an $\wa \sud$-bundle with $B \circ Q[\wa \sigma] \in H^2(X^\iota,\bZ)$. Then the set of embedding functors $E : \wa \sigma \to {\bf vect}(X^\iota)$ coincides with the set of vector bundles over $X^\iota$ of rank $d$ and first Chern class $B \circ Q[\wa \sigma]$.
\end{cor}

For the notion of {\em conjugate} in the setting of tensor \sC categories, we refer the reader to \cite[\S 2]{LR97}.
\begin{thm}
\label{thm_ci}
Let $d \in \bN$ and $G \subseteq \ud$ be covariant. Then for every $\wa G$-bundle $(\wa \rho , \otimes , \iota , \eps)$ the following invariants are assigned:
\[
\left\{
\begin{array}{ll}
\delta (\wa \rho)  
:=  
\delta \circ Q[\wa \rho] 
\in H^2( X^\iota , {G'} )  \ ,
&
\\
\breve \gamma_* (\wa \rho)  
:=  
\breve \gamma_* \circ Q[\wa \rho] 
\in 
\breve H^2(X^\iota,\aG)  \ .
&
{}
\end{array}
\right.
\]
The class $\breve \gamma_* (\wa \rho)$ defines a $G$-gerbe $\breve \mG$ over $X^\iota$, unique up to isomorphism, which collapses to a $G$-bundle $\mG$ if and only if there is an embedding functor $E : \wa \rho \to {\bf vect}(X^\iota)$, and in such a case
$\delta(\wa \rho) = [{\bf 1}]$. 
%
%
When $G \subseteq \sud$, the Chern class 
\[
c (\rho) \in H^2 ( X^\iota , \bZ ) \ ,
\]
defined in (\ref{def_cc_tso}), fulfilles the following properties: if $c(\rho) = 0$ then $\rho$ is a special object and the closure for subobjects of $\wa \rho$ has conjugates; if $E : \wa \rho \to {\bf vect}(X^\iota)$ is an embedding functor then $c(\rho)$ is the first Chern class of $E(\rho)$. 
\end{thm}

\begin{proof}
We pick a cocycle pair $\efb$ in the cohomology class 
$\breve \gamma_* (\wa \rho)$ 
and define $\breve \mG$ as the $G$-gerbe with transition maps defined by $\efb$ according to Rem.\ref{rem_def_gerbe} and Lemma \ref{lem_ddc}. 
Embeddings $E : \wa \rho \to {\bf vect}(X^\iota)$ are in one-to-one correspondence with $NG$-cocycles $\efn$ such that $p_*[\efn] = Q[\wa \rho]$ and the associated $G$-bundles $\mG$ define cohomology classes $\gamma_* [\efn] \in H^1(X^\iota,{\bf aut}G)$. Commutativity of the square in (\ref{eq_main}) implies that 
\[
\rmd_* \circ \gamma_*[\efn] 
\ = \ 
\breve \gamma_* \circ Q[\wa \rho] 
\ \in \breve H^2(X^\iota,\aG) \ ,
\]
and this proves that $\breve \mG$ is isomorphic to the gerbe defined by $\mG$ according to Rem.\ref{rem_def_gerbe}.
The relation between the existence of $E$ and the vanishing of $\delta(\wa \rho)$ is proved applying Thm.\ref{thm_emb} and Lemma \ref{lem_ddc}. 
%
%
%
Let now $G \subseteq \sud$. If $c(\rho) = 0$ then $\mR_\rho$ is a free Hilbert $(\ii)$-module and there is an isometry $S \in \mR_\rho$. So that $\rho$ is special, and \cite[Lemma.3.6]{DR89} implies that the conjugate $\ovl \rho$ is a subobject in $\wa \rho$. Using \cite[Thm.2.4]{LR97} we conclude that the tensor powers $\rho^r$, $r \in \bN$, and their subobjects, have conjugates.
\end{proof}

The previous theorem suggests that in general the dual object of a symmetric tensor \sC category is a nonabelian gerbe rather than a group bundle. Clearly, we should say in precise terms in which sense a tensor \sC category is the {\em representation category} of a gerbe. This could be done considering the notion of action of gerbes on bundles of $2$-Hilbert spaces. An alternative point of view is to consider Hilbert \sC bimodules rather than bundles: this situation is analogous to what happens in twisted $K$-theory, where we can use equivalently (Abelian) gerbes or bimodules with coefficients in a continuous trace \sC algebra to define the same $K$-group. These aspects will be clarified in a forthcoming paper (\cite{Vasf}).

\begin{ex}
\label{ex_su2}
Let $n \in \bN$ and $S^n$ denote the $n$-sphere. We discuss the map $p_*$ in the case $G = \sud$, $d > 1$:
\[
p_* : 
H^1 ( S^n , \ud ) 
\ \longrightarrow \
\sym (S^n , \wa \sud ) \simeq
H^1(S^n,\bT) \simeq
H^2(S^n,\bZ)
\ .
\]
A well-known argument implies $H^1 ( S^n , \ud ) \simeq \pi_{n-1}(\ud)$ (\cite[Ch.7.8]{Hus}); thus, by classical results (\cite[Ch.7.12]{Hus}) we have
\[
 H^1( S^2 , \ud ) \simeq H^1( S^4 , \ud ) \simeq \bZ  
 \ \ , \ \  
 H^1( S^1 , \ud ) \simeq H^1( S^3 , \ud ) \simeq {\bf 0}
 \ ;
\]
moreover,
\[
 H^2( S^2 , \bZ )   \simeq \bZ  
 \ \ , \ \  
 H^2( S^n , \bZ )   \simeq {\bf 0} \ , \ n \neq 2 \ .
\]
Thus the cases $S^1$, $S^3$ are trivial. In the other cases, we have the following:
\begin{itemize}
\item  $n>2$. The map $p_*$ is trivial and the unique element of $\sym ( S^n , \wa \sud )$ is the class of the trivial bundle. Now, it is a general fact that if $\mE \to X$ is a vector bundle, then the continuous bundle $( \mE , \mE )$ is trivial if and only if $\mE$ is the tensor product of a trivial bundle by a line bundle. In the case $X = S^n$, $n \neq 2$, every line bundle is trivial, thus we conclude that $( \mE , \mE )$ is trivial if and only if $\mE$ is trivial. Since $( \mE , \mE )$ is generated as a $C(X)$-module by the special unitary group of $\mE$, we conclude that $\mE \to S^n$ is trivial if and only if $\mcSUE \to S^n$ is trivial. Thus, $\wa \mcSUE$ is trivial for every $\mE \to S^n$, in spite of the fact that $\mcSUE$ is trivial if and only if $\mE = S^n \times \bC^d$. In particular, this holds for $S^{2m}$, $m = 2 , \ldots$, where nontrivial vector bundles exist.
\item $n=2$. We recall that the Chern character
\[
Ch : K^0(S^2) \to H^0 (S^2,\bZ) \oplus H^2 (S^2,\bZ)
\]
is a ring isomorphism. The term $H^0 (S^2,\bZ) \simeq$ $\bZ$ corresponds to the rank, whilst $H^2(S^2,\bZ) \simeq$ $\bZ$ is the first Chern class. By well-known stability properties of vector bundles (see \cite[Ch.8, Thm.1.5]{Hus} or \cite[II.6.10]{Kar}), we find that rank $d$ vector bundles $\mE , \mE' \to S^2$ are isomorphic if and only if $[\mE] = [\mE'] \in K^0(S^2)$, i.e. $Ch [\mE] = Ch [\mE']$. This implies that $p_*$ is one-to-one for $n=2$.
\end{itemize}
\end{ex}

\begin{ex}
\label{ex_rn}
We define $R_d \subset \sud$ as the group of diagonal matrices of the type $g := {\mathrm{diag}}(z , \ldots , z)$, where $z \in \bT$ is a root of unity of order $d$. Then $NR_d = \ud$ acts trivially on $R_d$ and $R_d' = R_d$. We have the exact sequence of pointed sets
\[
H^1(S^2,R_d) \stackrel{i_*}{\to} 
H^1(S^2,\ud) \stackrel{p_*}{\to} 
H^1(S^2,QR_d) \stackrel{\delta}{\to} 
H^2(S^2,R_d) 
\ .
\]
Now, every principal $R_d$-bundle over $S^2$ is trivial, and the universal coefficient theorem yields $H^2(S^2,R_d) \simeq {\mathrm{Hom}}(\bZ,R_d) \simeq \bZ_d$. Thus we have
\[
0 \to 
\bZ \stackrel{p_*}{\to} 
H^1(S^2,QR_d) \stackrel{\delta}{\to} 
\bZ_d
\ ,
\]
and $p_*$ is injective. We now prove that there is a left inverse $s : \bZ_d \to H^1(S^2,QR_d)$ for $\delta$ with trivial intersection with $p_*(\bZ)$. This suffices to prove that $H^1(S^2,QR^d) \simeq \bZ \oplus \bZ_d$. To this end, we embed $R_d$ in $\bT$ and regard each $R_d$-$2$-cocycle $\efg := \{ g_{ijk} \}$ as a $\bT$-$2$-cocycle. In this way, the argument of the proof of \cite[Thm.10.8.4(2)]{Dix} implies that there is an $1$-$\ud$-cochain $\efu := \{ u_{ij} \}$ such that $u_{ij} u_{jk} u_{ik}^{-1} = g_{ijk}$. Thus, we define the map
\[
s : H^2(S^2,\bT) \to \breve H^2 ( S^2 , R_d \lto \ud ) \simeq H^1(S^2,QR_d)
\ \ , \ \
s [\efg] := [ \efu , \efg ]
\ ,
\]
which clearly yields the desired left inverse (recall the definition of $\delta$). We conclude that
\[
\sym(S^2,\wa R_d) \simeq \bZ \oplus \bZ_d \ .
\]
The first direct summand corresponds to the term $H^1(S^2,\ud)$ whose isomorphism with $\bZ$ is realized by means of the determinant (see \cite[\S 7.8]{Hus}); this implies that the projection of $\sym(S^2,\wa R_d)$ on $\bZ$ is the Chern class. On the other side, by construction the projection on $\bZ_d$ corresponds to the class $\delta$.
\end{ex}

\begin{ex}
\label{ex_T}
Let $G \subset \ud$ be as in Ex.\ref{ex_T0} and $\wa \rho$ be a $\wa G$-bundle; then it is easy to check that the set of embeddings $\wa \rho \to {\bf vect}(X)$ is in one-to-one correspondence with vector bundles $\mE \to X$ such that $(\mE,\mE) \simeq (\rho,\rho)$.
In particular when $X$ is the $3$--sphere $S^3$ then every vector bundle $\mE \to S^3$ is trivial and $\wa \rho$ admits an embedding if and only if $(\rho,\rho)$ is trivial. On the other side, the (classical) Dixmier-Douady invariant is a complete invariant for bundles with fibre $\bM_d$ and base space $S^3$, thus we conclude that
\[
\delta : \sym(S^3,\wa G) \to H^2(S^3,G) = H^3(S^3,\bZ) = \bZ
\]
is an isomorphism.
\end{ex}

\noindent {\bf Acknowledgements.} The author would like to thank Mauro Spera for drawing his attention to gerbes, and an anonymous reviewer for suggesting several improvements on the first version of the present paper.


\end{document}